\documentclass[reqno]{amsart}
\theoremstyle{plain}
\usepackage{latexsym}
\usepackage{amssymb, color}
\usepackage{amsmath}
\usepackage{amsthm}
\usepackage{amsfonts}
\usepackage{mathrsfs}
\usepackage{enumerate}
\usepackage{epsfig}
\usepackage[sc]{mathpazo}
\def\XXint#1#2#3{{\setbox0=\hbox{$#1{#2#3}{\int}$}
     \vcenter{\hbox{$#2#3$}}\kern-.5\wd0}}

\newcommand{\cl}{\mathrm{curl}}
\newcommand{\Cl}{\nabla\times}
\newcommand{\ep}{\epsilon}
\newcommand{\mb}{\mathbb}
\newcommand{\R}{\mathbb{R}}

\newcommand{\lt}{\left}
\newcommand{\rt}{\right}
\newcommand{\g}{\mathcal{G}_{h_0}}
\newcommand{\h}{\check{H}^1(\mathbb R^3;\mathbb R^3)}
\newcommand{\supp}{\mathrm{supp}}
\newcommand{\ti}{\tilde}
\newcommand{\qd}{\quad}
\newcommand{\dist}{\mathrm{dist}}
\newcommand{\nl}{\newline}

\numberwithin{equation}{section}

\newtheorem{theorem}{Theorem}
\newtheorem{proposition}[theorem]{Proposition}
\newtheorem{lemma}[theorem]{Lemma}
\newtheorem{corollary}[theorem]{Corollary}

\title[First critical field of highly anisotropic $3$d superconductors]{First critical field of highly anisotropic three-dimensional superconductors via a vortex density model}
\author{Andres Contreras}
\address{Department of Mathematical Sciences\\New Mexico State University\\Las Cruces, NM 88003.}
\email{acontre@nmsu.edu.}
\author{Guanying Peng}
\address{Department of Mathematics\\University of Arizona\\Tucson, AZ 85721.}
\email{gypeng@math.arizona.edu.}

\begin{document}

\begin{abstract}

We analyze a mean field model for $3$d anisotropic superconductors with a layered structure, in the presence of a strong magnetic field. The mean field model arises as the $Gamma$-limit of the Lawrence-Doniach energy in certain regimes. A reformulation of the problem based on convex duality allows us to characterize the first critical field $H_{c_1}$ of the layered superconductor, up to leading order. In previous work, Alama-Bronsard-Sandier \cite{ABS} have derived the asymptotic value of $H_{c_1}$ for configurations satisfying periodic boundary conditions; in that setting describing minimizers of the Lawrence-Doniach energy reduces to a $2$d problem. In this work, we treat the physical case without any periodicity assumptions, and are thus led to studying a delicate and essentially $3$d non-local obstacle problem first derived by Baldo-Jerrard-Orlandi-Soner \cite{BJOS2} for the isotropic Ginzburg-Landau energy. We obtain a characterization of $H_{c_1}$ using the special anisotropic structure of the mean field model.
\end{abstract}

\maketitle

\section{Introduction} 

In this paper, we investigate a mean field model that describes the limiting behavior of a $3$d highly anisotropic cylindrical superconductor with layered structure. The state of the superconductor in response to an external magnetic influence is described at large scales in terms of a normalized vorticity. Our main goal is to characterize the asymptotic value of the applied field strength at which the sample transitions from a purely superconducting state to a mixed one where vortex defects appear in the interior.

The mathematical model for the anisotropic superconductor is the Lawrence-Doniach description. The layered structure in the Lawrence-Doniach functional can be observed in high temperature superconductors (e.g., the cuprates). Significant differences can be observed in the properties of these materials with respect to isotropic superconductors (for the latter type, the standard Ginzburg-Landau model is more suitable). Motivated by these differences, Lawrence and Doniach \cite{LD} proposed an alternate description where a layered anisotropic superconductor would not be treated as a continuous solid but as a stack of thin parallel superconducting layers. Mathematically, the layers interact through nonlinear Josephson coupling. Below, we recall the Lawrence Doniach model. The Josephson penetration depth $\lambda>0$ is a fixed constant that depends on the material.

Let $\Omega\subseteq\R^2$ be a smooth bounded simply-connected domain. Let $L>0$ be a constant and $N>0$ be an integer. Let $D$ be the $3$d cylindrical domain $D:=\Omega\times (0,L),$ and let $s:=L/N$ be the inter-layer distance. 
The stack of supercodunctors is subjected to an external field $h_{ex}\vec{e}_3 ,$ where $h_{ex}$ is the intensity of the field. The Lawrence-Doniach energy is given by
\begin{equation*}
\begin{split}
\mathcal{G}_{LD}^{\epsilon, s}(\{u_n\}_{n=0}^N, \vec{A})&= s\sum^N_{n=0} \int_\Omega\left[\frac{1}{2}|\hat{\nabla}_{\hat{A}_{n}}u_n|^2+\frac{(1-|u_n|^2)^2}{4\epsilon^2}\right]d\hat{x}\\
&+s\sum^{N-1}_{n=0}\int_\Omega\frac{1}{2\lambda^{2}s^2}\lt|u_{n+1}-u_{n}e^{\imath \int_{ns}^{(n+1)s}A^{3}dx_{3}}\rt|^{2}d\hat{x}\\
&+\frac{1}{2}\int_{\mathbb{R}^3}\lt|\nabla\times{\vec{A}}-h_{ex}\vec{e}_3\rt|^{2}dx.
\end{split}
\end{equation*}
Above, for each $n=0, \ldots , N,$ the function $u_n:\Omega\rightarrow\mb C$ corresponds to the wave map or order parameter, as in $2$d Ginzburg-Landau, of the $n$th layer. The effect of the applied magnetic field is made manifest in the induced potential $\vec{A}=(A^1,A^2,A^3):\mb R^3\rightarrow \mb R^3.$  The induced magnetic field is then given by $\nabla\times\vec{A} = (\partial_2 A^3-\partial_3 A^2,\partial_3 A^1-\partial_1 A^3,\partial_1 A^2-\partial_2 A^1)$. Other notations in the above energy are explained at the beginning of subsection \ref{s1.1}.

In phenomenological models of superconductors, the strength of an applied magnetic field has a great influence in the nature of minimizers of the energy.  More precisely, there are two critical values $H_{c_1}$ and $H_{c_3}$ of the field intensity $h_{ex}$ at which superconductors undergo phase transitions from the superconducting state to the mixed state (coexistence of superconducting and normal states), and from the mixed state to the normal state, respectively. In the London limit, that is when $\ep \to 0,$ these critical fields are expected to obey $H_{c_1}\sim|\ln\epsilon|$ and $H_{c_3}\sim\frac{1}{\epsilon^2}.$ Understanding the vortex structure of minimizers is of central importance when deriving asymptotics for the critical fields $H_{c_1}$ and $H_{c_3}.$  In $2$d Ginzburg-Landau, the behavior of minimizers and their vorticities, in different regimes of the strength of the applied field, is now well understood. For a detailed discussion where very precise asymptotics for the vorticity are derived, see the book \cite{SS}, the references therein and also \cite{SS1,JS, JerrardSpirn}.  For asymptotics valid near $H_{c_3},$ see \cite{GP, LP, HP, FH}. Configurations with a diverging number of vortices were analyzed in \cite{vorlatt} and in \cite{LVCS}; these correspond to global and local minimizers respectively.

In contrast with $2$d models, the $3$d situation is not as well understood. Recently, $\Gamma$-convergence results for the $3$d isotropic Ginzburg-Landau model in different energy regimes were obtained in \cite{BJOS1}. For a characterization of $H_{c_1}$ in $3$d valid for general domains, see \cite{BJOS2} (see also \cite{roman}). In \cite{ABM}, the authors constructed local minimizers (presumably global for certain ranges of the applied field) in a ball. Up to $o(1)$ asymptotics for $H_{c_1}$ are derived in \cite{co,cost} for thin superconductors. Finally, a characterization of the superconducting region for much higher values of the applied field in a superconducting shell is obtained in \cite{CLPersistence} based on a reduction to a double-sided obstacle problem. In general, a big problem in extending results from $2$d to $3$d lies in the description of the vorticity region which in the two dimensional case corresponds to a union of points, while in higher dimensions it can be given by very complex and nonsmooth structures. A notable challenge in deriving a more refined asymptotic expansion of the energy in $3$d is due to the fact that without a satisfactory description of vortices in this setting, an interaction energy of defects cannot be extracted. For a result in this direction see \cite{NPVF}.

	 Now, in what pertains to the $3$d anisotropic setting, more specifically for the Lawrence-Doniach energy, an analysis of minimizers for $h_{ex}$ in the regimes $h_{ex}\sim|\ln\ep|$ and $|\ln\ep|\ll h_{ex}\ll 1/\ep^2$ has been done by Alama-Bronsard-Sandier \cite{ABS} under certain periodicity assumptions. They also studied the cases when the magnetic fields are parallel to the layers or oblique in \cite{ABS} and \cite{ABS2}. Without the periodicity assumptions, a  great simplification to a mean field model in the form of a $\Gamma$-convergence result with $h_{ex}\sim|\ln\ep|$ is achieved by the second author in \cite{Pe}. In a higher regime, an asymptotic formula for the minimum Lawrence-Doniach energy with $|\ln\epsilon|\ll h_{ex}\ll\epsilon^{-2}$ in the limit as $(\ep,s)\rightarrow (0,0)$ is obtained in \cite{BP} together with information of the vortex structure. In the regime $h_{ex}\geq\frac{C}{\epsilon^2}$, it was shown by Bauman-Ko \cite{BK} that if $C$ is sufficiently large, all minimizers of the Lawrence-Doniach energy are in the normal phase. A similar result is known for the $2$d and $3$d Ginzburg-Landau energy (see \cite{GP}). Despite the explicit computations of the first critical field in the periodic case, a characterization of $H_{c_1}$ cannot be obtained using the same tools in the general case; under periodicity assumptions, additional structure is imposed that reduces the problem to a $2$d one. This reduction is not available in the situation contemplated here. The main goal of this paper is to investigate a mean field model that captures the limiting behavior of the Lawrence-Doniach energy when the intensity of the magnetic field is in the regime $h_{ex}\sim|\ln\epsilon|.$ Using a dual formulation, first derived for the isotropic Ginzburg-Landau energy \cite{BJOS2}, we obtain a first characterization of $H_{c_1}$ in terms of the solution to a non-local obstacle problem; this is in high contrast with the $2$d case in which the inverse of the maximum value of a Helmholtz equation gives the coefficient of the main order term for the first critical field. Even though the non-local nature of the equations obtained impedes an explicit expression for $H_{c_1}$ in the general $3$d isotropic case, we exploit the generalized cylindrical setting of the Lawrence-Doniach model and the particular form of its corresponding $\Gamma$-limit to give an explicit description of the intensity of the field that forces a nontrivial vorticity region in the sample. Our characterization is the first to provide an asymptotic for this value, valid in the physical case with natural boundary conditions (no periodicity assumptions.)

\subsection{Leading order of the first critical field via a non-local obstacle problem}\label{s1.1}

In order to present our main theorems, we introduce notation that will allow us to specify the contributions coming from variations within a layer (i.e. representing two dimensional quantities). To be precise, denote $\hat{x}:=(x_1,x_2)\mbox{ and }\hat{\nabla}:=(\partial_{1},\partial_{2}).$ For the magnetic potential, denote $\hat{A}:=(A^1,A^2),\mbox{ and  the trace of $\hat{A}$ on the $n$th layer by }\hat{A}_n(\hat{x}):=(A^{1}(\hat{x},ns),A^{2}(\hat{x},ns)).$
More generally, the notation $(\hat{\cdot})$ will be used for vectors and operators defined on $\Omega.$ In this way, $(\imath u,\hat{\nabla}u) \in \mathbb{R}^2$ is the vector with components $(\imath u,\partial_j u)$ for $j=1,2.$ Finally, the notation $(\vec{\cdot})$ will be reserved for three dimensional vectors.

 The standard tool to study the vorticity in Ginzburg-Landau is the Jacobian. In our discretized problem, this object can be decomposed as a sum of $2$d Jacobians. This reflects the intermediate character of the layered problem where both $2$d and $3$d features can be observed. The starting point is the $\Gamma$-convergence result in \cite{Pe} which reduces the problem to a mean field version of it for the current and the induced potential. After this, we use convex duality to get a formulation in the spirit of \cite{BJOS2} for the isotropic Ginzburg-Landau functional. From this we derive a novel, more explicit, characterization of nontrivial vorticity which yields a new expression of the first critical field in the Lawrence-Doniach model. 

For the $2$d Ginzburg-Landau energy, the Jacobian is the main tool for analyzing the vorticity. In the context of layered superconductors, the current and Jacobian are discrete objects defined by
\begin{equation*}
j^{\epsilon,s}(\{u_n\}_{n=0}^N)=\sum\limits_{n=0}^{N-1}j(u_n)\mathbb{1}_n(x_3), \quad J^{\epsilon,s}(\{u_n\}_{n=0}^N)=\sum\limits_{n=0}^{N-1}J(u_n)\mathbb{1}_n(x_3),
\end{equation*}
respectively, where $j(u_n):=(\imath u_n,\hat{\nabla}u_n)$ and $J(u_n):=\frac{1}{2}\cl j(u_n)$ are the $2$d current and Jacobian, respectively, and
\begin{equation*}
\mathbb{1}_n(x_3) = 
\begin{cases}
\mathbb{1}_{(0,s)}(x_3) & \text{ for } n=0,\\
\mathbb{1}_{[ns,(n+1)s)}(x_3) & \text{ for } n=1,\dddot\ ,N-1.
\end{cases}
\end{equation*}
The natural domain of definition of these is $[H^1(\Omega;\mathbb{C})]^{N+1}.$

In \cite{Pe}, under the assumptions that $\lim_{\epsilon\rightarrow 0}\frac{h_{ex}}{|\ln\epsilon|}=h_0$ for some $0\leq h_0<\infty$ and
$s|\ln\epsilon|\rightarrow\infty$ as $(\epsilon,s)\rightarrow (0,0)$, it is proved that $\mathcal{G}_{LD}^{\epsilon,s}$ Gamma-converges to
\begin{equation*}
\ti{\mathcal{G}}_{h_0}(\hat v,\vec A):=\frac{1}{2}\left[\lVert \hat v-\hat A \rVert_{L^2(D)}^2 + |\text{curl}\hat v|(D) + \lVert \nabla\times\vec A - h_0\vec e_3\rVert_{L^2(\mathbb R^3)}^2\right]
\end{equation*}
for a pair $(\hat v,\vec A)\in V\times \ti E_0$, where
\begin{equation}\label{V}
V:=\{\hat v\in L^2(D;\mb R^2): \cl \hat v \in \mathcal{M}(D)\},
\end{equation}
and
\begin{equation*}
\ti E_0:=\{\vec{C}\in H^{1}_{loc}(\mathbb{R}^3;\mathbb{R}^3):(\nabla\times\vec{C})-h_{0}\vec{e}_{3}\in L^{2}(\mathbb{R}^3;\mathbb{R}^3)\}.
\end{equation*}
In particular, minimizers $(\{u^{\epsilon}_n\},\vec{A}^{\epsilon,s})$  of $\mathcal{G}_{LD}^{\epsilon,s}$ satisfy

\[ \frac{j^{\epsilon,s}}{|\ln\epsilon|}\rightharpoonup \hat v \text{ in } L^{\frac{4}{3}}(D;\mathbb{R}^2),\qd
\frac{\vec A^{\epsilon,s} - h_{ex}\vec a}{|\ln\epsilon|} \rightharpoonup \vec A-h_0\vec a \text{\quad in } \check H^1(\mathbb R^3;\mathbb R^3),\]
where $(\hat v, \vec A)$ is a minimizer of $\ti{\mathcal{G}}_{h_0}$, $\vec{a}=\vec{a}(x)$ is any fixed smooth vector field on $\mathbb{R}^3$ such that $a^3=0$, $\nabla\times\vec{a}=\vec{e}_3$ and $\nabla\cdot \vec{a}=0$ in
	$\mathbb{R}^3$, and $\check{H}^1(\mathbb{R}^3;\mb R^3)$ is the completion of $C_0^{\infty}(\mathbb{R}^3;\mathbb{R}^3)$ with respect to the norm
\begin{equation}\label{eq501}
\lVert\vec{C}\rVert_{\check{H}^1(\mathbb{R}^3;\mb R^3)}=(\int_{\mathbb{R}^3}|\nabla\vec{C}|^2dx)^{\frac{1}{2}}.
\end{equation}

The main purpose of this paper is to characterize the critical value for $h_0$ below which minimizers of the $\Gamma$-limit functional $\ti{\mathcal{G}}_{h_0}$ satisfy $\cl \hat v = 0$, indicating that the vorticity measure vanishes for minimizers. As a result, the value of $H_{c_1}$ for the Lawrence-Doniach energy functional is obtained up to an $o(|\ln \epsilon |)$ error. Since our problem corresponds to a uniform applied magnetic field, it is more convenient to rescale the limiting functional $\ti{\mathcal{G}}_{h_0}$ by a factor $1/h_0^2$. Namely, we introduce the rescaled energy functional
\begin{equation*}
\begin{split}
&\mathcal{G}_{h_0}(\hat v,\vec{A}):= \frac{1}{h_0^2}\ti{\mathcal{G}}_{h_0}(h_0 \hat v, h_0 \vec{A}) \\
&\qd\qd= \frac{1}{2}\left[\lVert \hat v-\hat A \rVert_{L^2(D)}^2 + \frac{1}{h_0}|\cl\hat v|(D) + \lVert \nabla\times\vec A - \vec e_3\rVert_{L^2(\mathbb R^3)}^2\right]
\end{split}
\end{equation*}
and the corresponding admissible space for the magnetic potential
\begin{equation*}
E_0:=\{\vec{C}\in H^{1}_{loc}(\mathbb{R}^3;\mathbb{R}^3):(\nabla\times\vec{C})-\vec{e}_{3}\in L^{2}(\mathbb{R}^3;\mathbb{R}^3)\}.
\end{equation*}
It is clear that $(\hat v,\vec{A})\in V\times E_0$ minimizes $\mathcal{G}_{h_0}$ if and only if $(h_0 \hat v, h_0 \vec{A})\in V\times\ti E_0$ minimizes $\ti{\mathcal{G}}_{h_0}$. 
From \cite{BK}, each $\vec{C}\in\check{H}^1(\mathbb{R}^3;\mb R^3)$ has a representative in $L^6(\mathbb{R}^3;\mathbb{R}^3)$ such that
\begin{equation}\label{H11}
\lVert\vec{C}\rVert_{L^6(\mathbb{R}^3;\mathbb{R}^3)}\leq 2\lVert\vec{C}\rVert_{\check{H}^1(\mathbb{R}^3;\mb R^3)}.
\end{equation}
Further
\begin{equation}\label{H12}
\lVert\vec{C}\rVert_{\check{H}^1(\mathbb{R}^3;\mb R^3)}^2=\int_{\mathbb{R}^3}(|\nabla\cdot\vec{C}|^2+|\nabla\times\vec{C}|^2)dx.
\end{equation}
Define
\begin{equation*}
K_0:=\{\vec{C}\in E_0:\nabla\cdot\vec{C}=0 \text{ and } \vec{C}-\vec{a}\in \check{H}^1(\mathbb{R}^3;\mb R^3)\cap L^6(\mathbb{R}^3;\mathbb{R}^3)\},
\end{equation*}
which is the space for the divergence free Coulomb gauge of the magnetic potential $\vec A$ for $\mathcal{G}_{h_0}$. Existence of minimizers of $\mathcal{G}_{h_0}$ in $V\times K_0$ is a trivial consequence of the corresponding existence result for $\ti{\mathcal{G}}_{h_0}$ proved in \cite{Pe}. The first goal is to reformulate the minimization problem in terms of an obstacle problem: this turns out to be more convenient to capture the intensity of the applied field that forces $\cl v$ to be a nontrivial measure. Our first theorem accomplishes this and gives a dual equivalence to being a minimizer of $\mathcal{G}_{h_0}$. 

\begin{theorem}\label{t1}
A pair $(\hat v_0,\vec A_0)\in V\times K_0$ minimizes $\g$ if and only if the following two conditions are satisfied:
\begin{enumerate}
\item The vector field $\vec B_0:=\nabla\times(\vec A_0-\vec a)$ belongs to 
\begin{equation*}
\mathcal{C}_{h_0}:=\lt\{\vec B\in H^1(\mb R^3;\mb R^3)\cap\Cl\h: \supp(\nabla\times\vec{B})\subset \overline{D},\, \lVert \vec B\rVert_{*}\leq \frac{1}{2h_0}\rt\},
\end{equation*}
where
\begin{equation}\label{eq48}
\lVert\vec B\rVert_{*}:=\sup\lt\{\int_{\R^3}\vec{B}\cdot\lt(\Cl\vec{\phi}\rt)dx: \vec{\phi}\in H^1(\R^3;\R^3),\int_{D}|\cl\hat{\phi}|dx\leq 1\rt\}.
\end{equation}
In addition, denoting $\vec v_0=(\hat v_0, 0)\in\R^3$, we have
\begin{equation}\label{t1.2}
\nabla\times\vec B_0+\lt(\vec A_0-\vec v_0\rt)\mathbb{1}_{D} = 0 \text{ in } \mb R^3.
\end{equation}
\item $\vec B_0$ is the unique minimizer in $\mathcal C_{h_0}$ of the functional
\begin{equation*}
\mathcal{E}_{0}(\vec B):=\frac{1}{2}\int_{\mb R^3}\lt|\vec B\rt|^2dx+\frac{1}{2}\int_{D}\lt|\nabla\times\vec B+\vec a\rt|^2dx.
\end{equation*}
\end{enumerate}
\end{theorem}

The proof of Theorem \ref{t1} follows the analogous derivation in \cite{BJOS2} for the isotropic Ginzburg-Landau model. Here the $\lVert \cdot\rVert_{*}$ norm defined in (\ref{eq48}) differs from the one introduced in \cite{BJOS2} and can be viewed as an anisotropic analogue of the latter. We summarize some simple properties of minimizers of $\mathcal{G}_{h_0}$ which follow from Theorem \ref{t1}.

\begin{corollary}\label{c3}
Let $(\hat v_0,\vec A_0)\in V\times K_0$ be a minimizer of $\g$. Then we have $A_0^3=0$ and $(\Cl\vec{B}_0)^3=0$, where $(\Cl\vec{B}_0)^3=0$ denotes the $x_3$-component of $\Cl\vec{B}_0$. Moreover, if $\vec{w}_0\in L^2(\mb R^3;\mb R^3)$ is any vector field such that $\hat w_0|_{D}=(w_0^1|_D,w_0^2|_D)=\hat v_0$, then
\begin{equation}\label{t1.1}
\int_{\mb R^3}(\Cl\vec{B}_0)\cdot \vec{w}_0\,dx=-\frac{1}{2h_0}\lt|\cl \hat{w}_0\rt|(D).
\end{equation}
\end{corollary}

The next theorem gives a first characterization of the leading order of $H_{c_1}$, which is in the spirit of Theorem 3 in \cite{BJOS2}.

\begin{theorem}\label{t3}
Let $\hat v_0, \vec{A}_0, \vec{B}_0$ be as in Theorem \ref{t1}. Define the space $\mathcal{C}$ to be
\begin{equation*}
\mathcal{C}:=\lt\{\begin{split}&\vec B\in H^1(\mb R^3;\mb R^3)\cap\Cl\h: \int_{\R^3}\vec B\cdot(\nabla\times\vec{\phi})dx=0,\\ &\qd\qd\qd\qd\qd\qd\qd\qd\qd\qd\qd\forall \vec{\phi}\in H^1(\R^3;\R^3) \text{ s.t. } \cl\hat{\phi}=0 \text{ in } D\end{split}\rt\}.
\end{equation*}
Let $\vec B_*$ be the unique minimizer of $\mathcal{E}_{0}$ in $\mathcal{C}$. We have $\cl \hat v_0=0$ if and only if $\vec B_*=\vec B_0$ if and only if $\lVert\vec B_*\rVert_{*}\leq \frac{1}{2h_0}$.
\end{theorem}

The characterizations in Theorems \ref{t1} and \ref{t3} rely on minimizing the energy $\mathcal{E}_0$ subject to the constraint imposed by the $\lVert\cdot\rVert_{*}$ norm. In the $3$d setting, this is a non-local norm as opposed to the $L^{\infty}$ norm in $2$d, and is difficult to characterize in general. However, the highly anisotropic feature in our problem allows us to give a more explicit equivalent condition for $\cl \hat v_0=0$. 

\begin{theorem}\label{t4}
	Let $\vec B_*$ be the unique minimizer of $\mathcal{E}_{0}$ in $\mathcal{C}$. For all $x_3\in(0,L)$, let $\psi_{x_3}$ be the solution of the following problem
	\begin{equation}\label{t41}
	\begin{cases}
	-\hat\Delta\psi_{x_3}+(B_*^3(\cdot,x_3)+1)=0 &\text{in }\Omega,\\
	\psi_{x_3}=0 &\text{on } \partial\Omega,
	\end{cases}
	\end{equation}
	where $\hat\Delta$ is the two-dimensional Laplacian. Then $\cl \hat v_0=0$ if and only if $\lVert\psi_{x_3}\rVert_{\infty}\leq\frac{1}{2h_0}$ for all $x_3\in(0,L)$.
\end{theorem}

As a consequence of Theorem \ref{t4}, denoting $\xi:=\sup_{x_3\in\lt(0,L\rt)}\lVert\psi_{x_3}\rVert_{\infty}$, where $\psi_{x_3}$ is the solution of problem (\ref{t41}), we obtain the leading order expansion $H_{c_1}=\lt(\frac{1}{2\xi}+o(1)\rt)|\ln\ep|$ for the first critical field of the Lawrence-Doniach energy in the highly anisotropic regime $s|\ln\epsilon|\rightarrow\infty$.
To the best of our knowledge, this is the first time such an expression has been obtained for the full Lawrence-Doniach model with no simplifying assumptions of periodicity. Let us note that in the $3$d setting, explicit asymptotics for the value of the first critical field in terms of intrinsic geometric quantities are very hard to derive. For the isotropic model, the analogous expansions are available in the literature \cite{BJOS2, roman} in great generality but they depend on $\lVert\vec B_*\rVert_*$ and no further insight into this quantity is provided. Our characterization in Theorem \ref{t4} partially reduces the non-local norm to the $L^{\infty}$ norm of the functions $\psi_{x_3}$, although the functions $\psi_{x_3}$ still depend on $\vec{B}_{*}$ in a non-local way. Nevertheless, the vector field $\vec B_{*}$ is the minimizer of the energy functional $\mathcal{E}_0$ in the \emph{unconstrained} space $\mathcal{C}$. We expect that for certain domains with special symmetries, it is possible to write out the explicit expressions for $\vec B_{*}$. This is known to be true for spherical domains (see \cite{ABM}). If for certain cylindrical domains one can write out the explicit expression for $\vec B_{*}$, then the functions $\psi_{x_3}$ can be solved explicitly using the appropriate Green's function, and thus the leading order of $H_{c_1}$ can be made explicit for our problem. Our characterization is therefore a more complete description of $H_{c_1}$ in our setting.

Our paper is organized as follows. In the next section we gather some preliminary results that are needed for the subsequent characterizations of the first critical filed. In section \ref{section3} we use convex duality to derive the non-local obstacle problem for the measure $\cl v$ and the first properties of its corresponding minimizers. Later, in section \ref{section4} we prove Theorem \ref{t3}. Finally, in section \ref{section5} we obtain the more explicit characterization of triviality of the vorticity measure thus concluding the proof of Theorem \ref{t4}. An appendix is included at the end with the proof of a technical result about the regularity of double-sided obstacle problems that appear in our study.\nl

\noindent \textbf{Acknowledgments.} The first author was supported by a grant from the Simons Foundation \# 426318. The second author is very grateful to Wenhui Shi and Rohit Jain for helpful discussions on obstacle problems.

%
%

\section{Preliminaries}\label{s2}

In this section we gather some elementary results that will be needed later. We recall that $(\vec{\cdot})$ and $(\hat{\cdot})$ are reserved for three and two dimensional vectors respectively.  Additionally, if $\hat{w}$ is a two-dimensional vector, then $\vec{w}\in\R^3$ denotes $(\hat{w},0)=(w^1,w^2,0)$, and for $\vec w\in \R^3$, we denote by $\hat w=(w^1,w^2)\in\R^2$.

\begin{proposition}\label{p1}
	The minimizer of $\mathcal{E}_0$ is attained in the sets $\mathcal{C}_{h_0}$ and $\mathcal{C}$.
\end{proposition}

\begin{proof}
	We first show the existence of minimizer of $\mathcal{E}_0$ in the set $\mathcal{C}_{h_0}$. Let $\{\vec{B}_j\}_j\subset \mathcal{C}_{h_0}$ be a minimizing sequence of $\mathcal{E}_0$. Assume $\vec{B}_j=\Cl\vec\xi_j$ for $\vec\xi_j\in \h$. Without loss of generality, we may assume that $\nabla\cdot\vec{\xi}_j=0$ (see Lemma 3.1 in \cite{GP}), and hence, by (\ref{H12}) we have
	\begin{equation}\label{eq102}
	\|\vec\xi_j\|_{\h} = \|\vec{B}_j\|_{L^2(\R^3;\R^3)}.
	\end{equation}
	Since $\nabla\cdot\vec{B}_j=0$ and $\supp(\Cl\vec{B}_j)\subset \overline{D}$, it follows that
	\begin{equation}\label{eq101}
	\| \vec{B}_j\|_{H^1(\R^3;\R^3)}^2 = \int_{\R^3}\lt|\vec{B}_j\rt|^2 dx + \int_{D} \lt|\Cl\vec{B}_j\rt|^2 dx.
	\end{equation}
	Since $\{\vec{B}_j\}\subset \mathcal{C}_{h_0}$ is a minimizing sequence of $\mathcal{E}_0$, we deduce from (\ref{eq101}) that $\{\vec{B}_j\}$ forms a bounded sequence in $H^1(\R^3;\R^3)$, and hence has a weakly convergent subsequence which converges to some $\vec{B}_0\in H^1(\R^3;\R^3)$. On the other hand, using (\ref{eq102}), $\{\vec{\xi}_j\}$ forms a bounded sequence in $\h$, and hence, up to a subsequence, converges weakly in $\h$ to some $\vec\xi_0$. It is clear that $\Cl\vec{\xi}_0=\vec{B}_0$, and therefore $\vec{B}_0\in H^1(\R^3;\R^3)\cap \Cl\h$. By the Sobolev embedding theorem, we have $\vec{B}_j\rightarrow \vec{B}_0$ in $L^2_{loc}(\R^3;\R^3)$. Since $\supp(\Cl\vec{B}_j)\subset \overline{D}$, it follows that $\supp(\Cl\vec{B}_0)\subset \overline{D}$. Further, the $\|\cdot\|_{*}$ norm is preserved under strong $L^2$ convergence. Hence, $\vec B_0\in\mathcal{C}_{h_0}$. It follows from lower semicontinuity that $\vec B_0$ is the minimizer of $\mathcal{E}_0$ in $\mathcal{C}_{h_{0}}$. The existence of minimizer in $\mathcal{C}$ follows almost identical arguments.
\end{proof}

We will need the following convex duality result repeatedly, whose proof can be found, for example, in \cite{ekte}, Chapter IV.
\begin{lemma}\label{l2}
	Let $\Phi$ be convex lower semi-continuous from a Hilbert space $H$ to $(-\infty,\infty]$, and let $\Phi^*$ denote its conjugate, i.e.,
	\begin{equation}\label{eqp1}
	\Phi^*(f)=\sup_{g\in H}\lt(\langle f,g \rangle_{H}-\Phi(g)\rt),
	\end{equation}
	then
	\begin{equation*}
	\min_{u\in H}\lt(\frac{1}{2}\lVert u\rVert_H^2+\Phi(u)\rt)=-\min_{v\in H}\lt(\frac{1}{2}\lVert v\rVert_H^2+\Phi^*(-v)\rt)
	\end{equation*}
	and minimizers coincide.
\end{lemma}

Next we recall the following technical lemma from \cite{GP}.

\begin{lemma}[Lemma 3.1, \cite{GP}]\label{l8}
Let $\vec{g}\in L^2(\R^3;\R^3)$ be such that $\nabla\cdot\vec{g}=0$ in $\mathcal{D}'(\R^3)$. Then there is a unique $\vec{u}\in \check{H}^1(\R^3;\R^3)\cap L^6(\R^3;\R^3)$ such that $\nabla\times \vec{u}=\vec g$ and $\nabla\cdot\vec{u} = 0$.
\end{lemma}

Note that in the original statement of Lemma 3.1 in \cite{GP}, it is stated that $\vec{u}$ is unique in $\check{H}^1(\R^3;\R^3)$ instead of $\check{H}^1(\R^3;\R^3)\cap L^6(\R^3;\R^3)$. This is due to the slightly different definition of the space $\check{H}^1(\R^3;\R^3)$. Specifically, the space $\check{H}^1(\R^3;\R^3)$ defined in \cite{GP} corresponds to $\check{H}^1(\R^3;\R^3)\cap L^6(\R^3;\R^3)$ in this paper. Using the above lemma, we show

\begin{lemma}\label{l012}
	Given any $\vec{B}\in\mathcal{C}$, there exists a unique $\vec{A}\in K_0$ satisfying $\Cl\lt(\vec{A}-\vec{a}\rt)=\vec{B}$.
\end{lemma}

\begin{proof}
	Given $\vec{B}\in\mathcal{C}$, as $\vec{B}\in \nabla\times\check{H}^1(\R^3;\R^3)$, we have $\nabla\cdot\vec{B}=0$ in $\R^3$. By Lemma \ref{l8}, there exists $\vec{u}\in \check{H}^1(\R^3;\R^3)\cap L^6(\R^3;\R^3)$ such that $\nabla\times\vec{u}=\vec{B}$ and $\nabla\cdot\vec{u}=0$. Letting $\vec A=\vec u+\vec a$, we have $\vec A - \vec a\in \check{H}^1(\R^3;\R^3)\cap L^6(\R^3;\R^3)$ and $\nabla\times(\vec A - \vec a) = \vec B$, $\nabla\cdot(\vec A - \vec a)=0$. It follows that $\vec A \in K_0$. 
	
	Let $\vec A_1\in K_0$ be such that $\nabla\times(\vec A_1-\vec a) = \vec B$. Denoting $\vec u_1=\vec A_1-\vec a$, it follows from the definition of $K_0$ that $\vec u_1\in \check{H}^1(\R^3;\R^3)\cap L^6(\R^3;\R^3)$ and $\nabla\times \vec u_1=\vec B$, $\nabla\cdot\vec u_1=0$.  We conclude from Lemma \ref{l8} that $\vec u_1=\vec u$ and thus $\vec A_1=\vec A$. This shows the uniqueness of $\vec A$.
\end{proof}

%
%

\section{Characterization of minimizers of $\mathcal{G}_{h_0}$: proof of Theorem \ref{t1} and Corollary \ref{c3}}\label{section3}

We start with the proof of Theorem \ref{t1}, which relies on the convex duality result stated in Lemma \ref{l2} and follows closely the calculations in the proof of Theorem 2 in \cite{BJOS2}. Here some subtle modifications are needed to account for the highly anisotropic features in our problem. We define the space $\check{H}^1_{div}(\mb R^3;\mb R^3)$ to be
\begin{equation}\label{eq510}
\check{H}^1_{div}(\mb R^3;\mb R^3):=\{\vec C\in\check{H}^1(\mb R^3;\mb R^3):\nabla\cdot\vec C=0 \}.
\end{equation}
This is a Hilbert space with the inner product 
\begin{equation}\label{eq511}
\lt(\vec\phi,\vec{\psi}\rt)_{\check{H}^1_{div}(\mb R^3;\mb R^3)}:= \int_{\R^3}\lt(\Cl\vec{\phi}\rt)\cdot\lt(\Cl\vec{\psi}\rt)\, dx.
\end{equation}
Since $\check{H}^1_{div}(\R^3;\R^3)$ is a closed subspace of $\check{H}^1(\R^3;\R^3)$, we have the decomposition $\check{H}^1= \check{H}^1_{div}\oplus(\check{H}^1_{div})^{\perp}$. We need a simple characterization of $(\check{H}^1_{div})^{\perp}$. First we note the following fact whose proof is standard. We include the proof for completeness. 

\begin{lemma}\label{l011}
The space $C^{\infty}_c(\R^3;\R^3)$ is dense in $\check{H}^1(\R^3;\R^3)$ with respect to the norm defined in (\ref{eq501}).
\end{lemma}

\begin{proof}
By definition, the space $C^{\infty}_0(\R^3;\R^3)\cap \check{H}^1(\R^3;\R^3)$ is dense in $\check{H}^1(\R^3;\R^3)$. Therefore we only need to show that $C^{\infty}_c(\R^3;\R^3)$ is dense in $C^{\infty}_0(\R^3;\R^3)\cap \check{H}^1(\R^3;\R^3)$. Let $\eta(x)\in C^{\infty}_c(\R^3)$ be a standard cutoff function such that $\eta(x)=1$ near the origin. For any $\vec\psi\in C^{\infty}_0(\R^3;\R^3)\cap \check{H}^1(\R^3;\R^3)$ and $R>0$, define $\vec{\psi}_R(x):=\vec{\psi}(x)\eta_R(x)$, where $\eta_R(x)=\eta(\frac{x}{R})$. It is clear that $\vec{\psi}_R \in C^{\infty}_c(\R^3;\R^3)$. Further, we have
\begin{equation*}
\int_{\R^3}\lt|\nabla\vec{\psi}-\nabla\vec{\psi}_R\rt|^2\,dx \leq 2\int_{\R^3}\lt|\nabla\vec{\psi}\rt|^2\lt|1-\eta_R\rt|^2\,dx + 2\int_{\R^3}\lt|\vec{\psi}\rt|^2\lt|\nabla\eta_R\rt|^2\,dx.
\end{equation*}
Since $\nabla\vec{\psi}\in L^2(\R^3)$, it is clear that the above first term on the right hand side tends to zero as $R\rightarrow \infty$. For the second term on the right hand side, it follows from H\"{o}lder's inequality that
\begin{equation}\label{eq505}
2\int_{\R^3}\lt|\vec{\psi}\rt|^2\lt|\nabla\eta_R\rt|^2\,dx\leq 2\lt(\int_{\supp(\nabla\eta_R)}|\vec{\psi}|^6\,dx\rt)^{\frac{1}{3}} \lt(\int_{\supp(\nabla\eta_R)}\lt|\nabla\eta_R\rt|^3\,dx\rt)^{\frac{2}{3}}.
\end{equation}
It follows from (\ref{H11}) that 
\begin{equation}
\int_{\supp(\nabla\eta_R)}|\vec{\psi}|^6\,dx \rightarrow 0 \text{ as } R\rightarrow\infty.
\end{equation}
On the other hand, setting $y=\frac{x}{R}$, we have
\begin{equation}\label{eq506}
\int_{\supp(\nabla\eta_R)}\lt|\nabla\eta_R\rt|^3\,dx = \int_{\R^3}\lt|\nabla_x\lt(\eta\lt(\frac{x}{R}\rt)\rt)\rt|^3\,dx= \int_{\R^3}\lt|\nabla_y\lt(\eta\lt(y\rt)\rt)\rt|^3\,dy= \lVert \nabla\eta\rVert_{L^3(\R^3)}^3.
\end{equation}
Putting (\ref{eq505})-(\ref{eq506}) together, we obtain
\begin{equation*}
2\int_{\R^3}\lt|\vec{\psi}\rt|^2\lt|\nabla\eta_R\rt|^2\,dx \rightarrow 0 \text{ as } R\rightarrow\infty.
\end{equation*}
It follows that $\lVert \nabla\vec{\psi}-\nabla\vec{\psi}_R\rVert_{L^2(\R^3)}\rightarrow 0$ as $R\rightarrow \infty$ and hence the set $C^{\infty}_c(\R^3;\R^3)$ is dense in $C^{\infty}_0(\R^3;\R^3)\cap \check{H}^1(\R^3;\R^3)$.
\end{proof}

\begin{lemma}\label{l010}
For any $\vec\psi\in(\check{H}^1_{div})^{\perp}$, we have $\Cl\vec{\psi}=0$.
\end{lemma}

\begin{proof}
For any $\vec\psi\in(\check{H}^1_{div})^{\perp}$, by Lemma \ref{l011} and (\ref{H12}), there exists a sequence $\{\vec\psi_k\}\subset C^{\infty}_c(\R^3;\R^3)$ such that
\begin{equation*}
\int_{\R^3}\lt(\lt|\Cl(\vec{\psi}_k-\vec{\psi}) \rt|^2+\lt|\nabla\cdot(\vec{\psi}_k-\vec{\psi})\rt|^2\rt)\,dx \rightarrow 0.
\end{equation*}
In particular, we have 
\begin{equation*}
\int_{\R^3}\lt|\Cl(\vec{\psi}_k-\vec{\psi}) \rt|^2\,dx \rightarrow 0.
\end{equation*}
For each $\vec{\psi}_k$, by standard Hodge decomposition, we have $\vec{\psi}_k=\vec{\psi}_k^1+\vec{\psi}_k^2$, where $\nabla\cdot\vec{\psi}_k^1=0$, $\nabla\times\vec{\psi}_k^2=0$ and $\vec{\psi}_k^j, \nabla\vec{\psi}_k^j \in L^2(\R^3)\cap C^{\infty}(\R^3)$ for $j=1,2$ (see, e.g., Proposition 1.16 in \cite{MaBe}). Therefore, we have $\vec\psi_k^1\in \check{H}^1_{div}(\R^3;\R^3)$ and hence, noting $\vec\psi\in(\check{H}^1_{div})^{\perp}$,
\begin{equation*}
0=\lt(\vec\psi_k^1,\vec\psi\rt)_{\check{H}^1}=\int_{\R^3}\lt(\Cl\vec{\psi}_k^1\rt)\cdot\lt(\Cl\vec{\psi}\rt)\,dx
\end{equation*}
for all $k$. It follows that
\begin{equation*}
\begin{split}
\int_{\R^3}\lt|\Cl(\vec{\psi}_k-\vec{\psi}) \rt|^2\,dx&=\int_{\R^3}\lt|\Cl(\vec{\psi}_k^1-\vec{\psi}) \rt|^2\,dx\\
&=\int_{\R^3}\lt(\lt|\Cl\vec{\psi}_k^1\rt|^2+\lt|\Cl\vec{\psi}\rt|^2\rt)\,dx\rightarrow 0,
\end{split}
\end{equation*}
from which we conclude that $\Cl\vec{\psi}=0$.
\end{proof}

\begin{proof}[Proof of Theorem \ref{t1}]
	We denote $\hat\xi:=\hat v-\hat A\mathbb{1}_{D}$ and $\vec\zeta:=\vec A-\vec a$. Let $H:=L^2(D;\mb R^2)\times\check{H}^1_{div}(\mb R^3;\mb R^3)$, where $\check{H}^1_{div}(\mb R^3;\mb R^3)$ is defined in (\ref{eq510}) with the inner product given in (\ref{eq511}). Then $H$ is a Hilbert space with the inner product
	\begin{equation*}
	\lt((\hat\xi_1,\vec\zeta_1),(\hat\xi_2,\vec\zeta_2)\rt)_H=\int_{D} \hat\xi_1\cdot \hat\xi_2 \, dx + \int_{\R^3} \lt(\Cl\vec\zeta_1\rt) \cdot \lt(\Cl\vec\zeta_2\rt) \, dx
	\end{equation*}
	and the norm 
	\begin{equation*}
	\lVert (\hat\xi,\vec\zeta)\rVert_{H}^2=\lVert\hat\xi\rVert_{L^2(D)}^2+\lVert\Cl\vec\zeta\rVert_{L^2(\R^3)}^2.
	\end{equation*}
	We rewrite $\g(\hat v,\vec A)$ as
	\begin{equation*}
	\g(\hat v,\vec A)=\mathcal{F}(\hat{\xi},\vec{\zeta}):=\frac{1}{2}\lVert(\hat{\xi},\vec{\zeta})\rVert_{H}^2+\Phi\lt((\hat{\xi},\vec{\zeta})\rt),
	\end{equation*}
	where $\Phi\lt((\hat{\xi},\vec{\zeta})\rt)=\frac{1}{2h_0}\lt|\cl\lt(\hat{\xi}+\hat{\zeta}+\hat a\rt)\rt|(D)$ and $\lt|\cl\lt(\hat{\xi}+\hat{\zeta}+\hat a\rt)\rt|(D)$ is the total variation of the measure $\cl\lt(\hat{\xi}+\hat{\zeta}+\hat a\rt)$. By convention, $\Phi\lt((\hat{\xi},\vec{\zeta})\rt)$ is understood to equal $+\infty$ if $\cl\lt(\hat{\xi}+\hat{\zeta}+\hat a\rt)$ fails to be a finite Radon measure. It is straightforward to check that $\Phi$ is convex and lower semi-continuous. 

	First we compute the conjugate $\mathcal{F}^*$ of $\mathcal{F}$ given by
	\begin{equation}\label{eqq1}
	\mathcal{F}^*(\hat{\xi},\vec{\zeta}):=\frac{1}{2}\lVert(\hat{\xi},\vec{\zeta})\rVert_{H}^2+\Phi^*\lt(-(\hat{\xi},\vec{\zeta})\rt),
	\end{equation}
	where $\Phi^*$ is the conjugate of $\Phi$ computed according to (\ref{eqp1}). For $(\hat\xi,\vec\zeta)\in H$, we compute
	\begin{equation*}
	\begin{split}
	\Phi^*\lt((\hat{\xi},\vec{\zeta})\rt)&=\sup_{(\hat{\phi},\vec{\psi})\in H}\lt(\int_{D}\hat{\xi}\cdot\hat{\phi}dx+\int_{\mb R^3}(\Cl\vec{\zeta})\cdot(\Cl\vec{\psi})dx-\frac{1}{2h_0}\lt|\cl\lt(\hat{\phi}+\hat{\psi}+\hat a\rt)\rt|(D)\rt)\\
	=\sup_{(\hat{\phi},\vec{\psi})\in H}&\lt(\int_{D}\hat{\xi}\cdot\lt(\hat{\phi}+\hat a\rt)dx+\int_{\mb R^3}(\Cl\vec{\zeta})\cdot(\Cl\vec{\psi})dx-\frac{1}{2h_0}\lt|\cl\lt(\hat{\phi}+\hat{\psi}+\hat a\rt)\rt|(D)\rt)\\
	&\quad-\int_{D}\hat{\xi}\cdot \hat a\, dx\\
	=\sup_{(\hat{\phi},\vec{\psi})\in H}&\lt(\int_{\mb R^3}\lt(\hat{\xi}\cdot\hat{\phi}\mathbb{1}_{D}+(\Cl\vec{\zeta})\cdot(\Cl\vec{\psi})\rt)dx-\frac{1}{2h_0}\lt|\cl\lt(\hat{\phi}+\hat{\psi}\rt)\rt|(D)\rt)-\int_{D}\hat{\xi}\cdot \hat a \,dx.
	\end{split}
	\end{equation*}
	By homogeneity, it is clear that the above supremum in the above last line equals zero if
	\begin{equation}\label{t11}
	\int_{\mb R^3}\lt(\hat{\xi}\cdot\hat{\phi}\mathbb{1}_{D}+(\Cl\vec{\zeta})\cdot(\Cl\vec{\psi})\rt)dx\leq\frac{1}{2h_0}\lt|\cl\lt(\hat{\phi}+\hat{\psi}\rt)\rt|(D) \text{ for all } (\hat\phi,\vec{\psi})\in H,
	\end{equation}
	and it equals infinity if \eqref{t11} fails. It follows that
	\begin{equation*}
	\Phi^*\lt((\hat{\xi},\vec{\zeta})\rt)=
	\begin{cases}
	-\int_{D}\hat{\xi}\cdot \hat a \,dx &\text{ if \eqref{t11} holds},\\
	+\infty &\text{ otherwise},
	\end{cases}
	\end{equation*}
	and thus, by (\ref{eqq1}),
	\begin{equation}\label{t12}
	\mathcal{F}^*(\hat{\xi},\vec{\zeta})=\begin{cases}
	\frac{1}{2}\lVert(\hat{\xi},\vec{\zeta})\rVert_{H}^2+\int_{D}\hat{\xi}\cdot\hat a\,dx &\text{ if \eqref{t11} holds},\\
	+\infty &\text{ otherwise}.
	\end{cases}
	\end{equation}
	
	Now we show that \eqref{t11} is equivalent to the following two conditions
	\begin{equation}\label{t13}
	\int_{\mb R^3}\lt(\Cl\vec\zeta\rt)\cdot\lt(\Cl\vec\psi\rt) dx\leq \frac{1}{2h_0}\int_{D}\lt|\cl\hat{\psi}\rt|dx \text{ for all } \vec{\psi}\in \h,
	\end{equation}
	and
	\begin{equation}\label{t14}
	\vec\zeta\in H^2_{loc}\cap \check{H}^1_{div} \qd\text{ and }\qd \Delta\vec\zeta + \vec\xi\mathbb{1}_{D} = 0 \text{ a.e. in } \R^3,
	\end{equation}
	where recall that $\vec\xi=(\hat\xi,0)$. First, assume that \eqref{t11} holds. Given $\vec{\psi}\in\h$, we write $\vec{\psi}=\vec{\psi}_1+\vec{\psi}_2$ such that $\vec{\psi}_1\in\check{H}^1_{div}$ and $\vec{\psi}_2\in(\check{H}^1_{div})^{\perp}$. By Lemma \ref{l010}, we have $\Cl\vec{\psi}_2=0$. For all $(\hat\phi,\vec\psi)\in L^2(D;\R^2)\times \check{H}^1(\R^3;\R^3)$, using the above decomposition and (\ref{t11}), we have
	\begin{equation}\label{eqa1}
	\begin{split}
	&\int_{\R^3}\lt(\hat\xi\cdot\hat\phi\mathbb{1}_{D}+(\Cl\vec{\zeta})\cdot(\Cl\vec{\psi})\rt) dx\\
	 &\qd\qd\qd= \int_{\R^3}\lt(\hat\xi\cdot\hat\phi\mathbb{1}_{D}+(\Cl\vec{\zeta})\cdot(\Cl\vec{\psi}_1)\rt) dx\\
	&\qd\qd\qd \leq \frac{1}{2h_0}\int_{D} \lt|\cl\lt(\hat\phi+\hat{\psi}_1\rt)\rt| dx = \frac{1}{2h_0}\int_{D} \lt|\cl\lt(\hat\phi+\hat{\psi}\rt)\rt| dx.
	\end{split}
	\end{equation}
	Taking $\hat{\phi}\equiv 0$ and $\vec\psi\in\h$ in (\ref{eqa1}), we obtain
	\begin{equation*}
	\int_{\R^3}(\Cl\vec{\zeta})\cdot(\Cl\vec{\psi}) dx\leq \frac{1}{2h_0}\int_{D} \lt|\cl\hat{\psi}\rt| dx\text{ for all } \vec{\psi}\in \h,
	\end{equation*}
	which is \eqref{t13}. Next, by taking $\vec{\psi}\in \check{H}^1(\R^3;\R^3)$ and $\hat{\phi}=-\hat{\psi}\mathbb{1}_{D}$ in \eqref{eqa1}, we obtain
	\begin{equation}\label{eq31}
	\int_{\R^3}\lt(-\hat{\xi}\cdot\hat{\psi}\mathbb{1}_{D}+(\Cl\vec{\zeta})\cdot(\Cl\vec{\psi})\rt)dx = 0 \text{ for all } \vec\psi\in\check{H}^1(\R^3;\R^3).
	\end{equation}
	In  particular, (\ref{eq31}) holds for all $\vec{\psi}\in C^{\infty}_c(\R^3;\R^3)$. Direct calculations using integration by parts and the fact that $\nabla\cdot\vec{\zeta}=0$ yield
	\begin{equation*}
	\int_{\R^3}(\Cl\vec\zeta)\cdot(\Cl\vec\psi)\,dx=\int_{\R^3}(\nabla\vec\zeta)\cdot(\nabla\vec\psi)dx \text{ for all } \vec\psi\in C^{\infty}_c(\R^3;\R^3)
	\end{equation*}
	where $(\nabla\vec{\zeta})\cdot(\nabla\vec{\psi})=\sum_{j=1}^{3}\nabla\zeta^j \cdot \nabla\psi^j$, and thus
	\begin{equation*}
	\int_{\R^3}\lt(-\hat{\xi}\cdot\hat{\psi}\,\mathbb{1}_{D}+(\nabla\vec{\zeta})\cdot(\nabla\vec{\psi})\rt)dx = 0 \text{ for all } \vec\psi\in C^{\infty}_c(\R^3;\R^3).
	\end{equation*}
	It follows that $-\vec{\xi}\,\mathbb{1}_{D}-\Delta\vec{\zeta} = 0$ in the weak sense in $\R^3$. By standard elliptic regularity (see, e.g., \cite{GT}), we have that $\vec{\zeta}\in H^2_{loc}(\R^3;\R^3)$ and hence we have (\ref{t14}).
	
	Conversely, assume that \eqref{t13} and \eqref{t14} hold. Given $(\hat{\phi},\vec{\psi})\in H$, if $\cl\hat{\phi}$ fails to be a finite Radon measure, then (\ref{t11}) is trivially satisfied. Therefore, we may assume without loss of generality that $\hat{\phi}\in V$, where recall that the space $V$ is defined in (\ref{V}). We will need the following technical lemma:
	
	\begin{lemma}\label{l7}
	Let $(\hat{\phi},\vec{\psi})\in V\times \check{H}^1_{div}$. Then there exists a sequence $\{(\hat{\phi}_k,\vec{\psi}_k)\}_k\subset C_{c}^{\infty}(\R^3;\R^2)\times C_0^{\infty}(\R^3;\R^3)$ with the following properties:
	\begin{equation}\label{eq10}
	\hat{\phi}_k|_D \rightarrow \hat{\phi} \text{ in } L^2(D),\qd \Cl\vec{\psi}_k \rightharpoonup \Cl\vec{\psi} \text{ in } L^2(\R^3),
	\end{equation} 
	and
	\begin{equation}\label{eq11}
	|\cl(\hat{\phi}_k+\hat{\psi}_k)|(D) \rightarrow |\cl(\hat{\phi}+\hat{\psi})|(D).
	\end{equation}
	\end{lemma}
	
	We postpone the proof of Lemma \ref{l7} to the end of this section. Now for given $(\hat{\phi},\vec{\psi})\in V\times \check{H}^1_{div}$, let $\{(\hat{\phi}_k,\vec{\psi}_k)\}\subset C_{c}^{\infty}(\R^3;\R^2)\times C_0^{\infty}(\R^3;\R^3)$ be the sequence found in Lemma \ref{l7} satisfying the properties (\ref{eq10})-(\ref{eq11}). We deduce from \eqref{t13} and \eqref{t14} that
	\begin{equation*}
	\begin{split}
	\int_{\mb R^3}\lt(\vec{\xi}\cdot\vec\phi_k\,\mathbb{1}_{D}+(\Cl\vec{\zeta})\cdot(\Cl\vec{\psi}_k)\rt)dx&=\int_{\mb R^3}\lt(-\Delta\vec{\zeta}\cdot\vec\phi_k+(\Cl\vec{\zeta})\cdot(\Cl\vec{\psi}_k)\rt)dx\\
	&=\int_{\mb R^3}(\Cl\vec{\zeta})\cdot\lt(\Cl(\vec{\phi}_k+\vec{\psi}_k)\rt)dx\\
	&\leq \frac{1}{2h_0}\int_{D}\lt|\cl(\hat{\phi}_k+\hat{\psi}_k)\rt|dx.
	\end{split}
	\end{equation*}
	Therefore, passing to the limit as $k\rightarrow \infty$ and using (\ref{eq10})-(\ref{eq11}), we conclude that $\eqref{t11}$ holds for all $(\hat{\phi},\vec{\psi})\in V\times\check{H}^1_{div}$.
	
	Recall the expression for $\mathcal{F}^*$ in (\ref{t12}). When $\mathcal{F}^*$ is finite, the condition (\ref{t11}) is satisfied and thus $\vec{\zeta}\in H^2_{loc}$. Direct calculations using $\nabla\cdot\vec{\zeta}=0$ give $-\Delta\vec{\zeta}=\Cl(\Cl\vec{\zeta})$. Rewriting $\mathcal{F}^*$ using $\vec{B}=\Cl(\vec A-\vec a)=\Cl\vec\zeta$ and (\ref{t14}), we have
	\begin{equation*}
	\begin{split}
	\mathcal{F}^*(\hat{\xi},\vec{\zeta})&=\frac{1}{2}\int_{D}\lt|\hat{\xi}\rt|^2dx + \frac{1}{2}\int_{\mb R^3}\lt|\Cl\vec\zeta\rt|^2dx+\int_{D}\hat{\xi}\cdot\hat{a}\,dx\\
	&=\frac{1}{2}\int_{D}\lt|\hat{\xi}+\hat a\rt|^2dx + \frac{1}{2}\int_{\mb R^3}\lt|\Cl\vec\zeta\rt|^2dx-\frac{1}{2}\int_{D}\lt|\hat{a}\rt|^2dx\\
	&=\frac{1}{2}\int_{D}\lt|\Cl\Cl\vec\zeta+\vec a\rt|^2dx + \frac{1}{2}\int_{\mb R^3}\lt|\Cl\vec\zeta\rt|^2dx-\frac{1}{2}\int_{D}\lt|\hat{a}\rt|^2dx\\
	&=\frac{1}{2}\int_{D}\lt|\Cl\vec B+\vec a\rt|^2dx + \frac{1}{2}\int_{\mb R^3}\lt|\vec B\rt|^2dx-\frac{1}{2}\int_{D}\lt|\hat{a}\rt|^2dx,
	\end{split}
	\end{equation*}
	where in the above we have used $a^3=0$. According to Lemma \ref{l2}, denoting $\hat\xi_0=\hat v_0-\hat A_0\mathbb{1}_{D}$ and $\vec\zeta_0=\vec A_0-\vec a$, we have
	\begin{equation*}
	\begin{split}
	(\hat v_0,\vec A_0) \text{ minimizes } \g&\Longleftrightarrow (\hat\xi_0,\vec{\zeta}_0) \text{ minimizes } \mathcal{F}\\
	&\Longleftrightarrow (\hat\xi_0,\vec{\zeta}_0) \text{ minimizes } \mathcal{F}^*\\
	&\Longleftrightarrow \vec B_0 \text{ minimizes } \mathcal{E}_{0},
	\end{split}
	\end{equation*}
	provided that \eqref{t13} and \eqref{t14} are satisfied. This completes the proof of Theorem \ref{t1}.
\end{proof}

\begin{proof}[Proof of Corollary \ref{c3}]
Let $(\hat v_0,\vec{A}_0)\in V\times K_0$ be a minimizer of $\mathcal{G}_{h_0}$. By Theorem \ref{t1}, we know that $\vec{B}_0=\Cl(\vec{A}_0-\vec{a})\in\mathcal{C}_{h_0}$. In particular, we have
\begin{equation*}
\int_{\mb R^3}\vec B_0\cdot\lt(\nabla\times\vec{\psi}\rt)dx\leq\frac{1}{2h_0}\int_{D}|\cl\hat{\psi}|dx \text{ for all }\vec\psi\in H^1(\mb R^3;\mb R^3).
\end{equation*}
Taking $\vec{\psi}=(0,0,\pm\psi^3)$ with $\psi^3\in C^{\infty}_{c}(\R^3)$ in the above and noting that $\supp(\Cl\vec{B}_0)\subset \overline{D}$, we obtain
\begin{equation*}
0=\int_{\mb R^3}\vec B_0\cdot\lt(\nabla\times\vec{\psi}\rt)dx = \int_{D}(\Cl\vec B_0)\cdot\vec{\psi}\,dx = \int_{D}(\Cl\vec B_0)^3\,\psi^3\,dx\text{ for all } \psi^3\in C^{\infty}_{c}(\R^3),
\end{equation*}
where we denote by $(\Cl\vec B_0)^3$ the $x_3$-component of the vector $\Cl\vec B_0$. Taking a sequence $\{\phi^3_k\}\subset C_c^{\infty}(\R^3)$ that converges strongly to $(\Cl\vec B_0)^3$ in $L^2(D)$, we conclude that $(\Cl\vec B_0)^3 = 0$ in $\R^3$. As $\Cl \vec B_0=\Cl(\Cl(\vec A_0-\vec a))=-\Delta\vec A_0$, we have $-\Delta A_0^3=0$ in $\R^3$. Since $A_0^3-a^3\in L^6(\R^3)$ and $a^3=0$, the maximum principle implies that $A_0^3 = 0$ in $\R^3$ as desired.

To obtain (\ref{t1.1}), we use the fact that $\frac{d}{dt}\mathcal{G}_{h_0}(\hat v_0e^t,\vec{A}_0)|_{t=0} = 0$ to deduce 
	\begin{equation*}
	\int_{D}\lt(\hat v_0-\hat{A}_0\rt)\cdot \hat v_0\, dx + \frac{1}{2h_0}\lt|\cl \hat v_0\rt|(D) = 0.
	\end{equation*}
	It is clear that (\ref{t1.1}) follows from this and (\ref{t1.2}).
\end{proof}

We conclude this section with the proof of Lemma \ref{l7}. This result is similar to Proposition 2.3 in \cite{Pe} and the proof is provided in detail there. Here we provide another proof that adapts the proof of Lemma 15 in \cite{BJOS2} to our anisotropic setting. We provide the details for the benefits of later discussions and for the convenience of the readers. 

\begin{proof}[Proof of Lemma \ref{l7}]
Recall that $D=\Omega\times(0,L)$, where $\Omega\subset\R^2$ is a bounded smooth domain. Given sufficiently small $\ep>0$, we denote 
\begin{equation*}
\Omega_{\ep}:= \{\hat{x}\in\R^2: \dist(\hat{x},\Omega)< \ep\} \qd\text{ and }\qd \Omega_{-\ep}:= \{\hat{x}\in\Omega: \dist(\hat x,\partial\Omega)>\ep \}.
\end{equation*}
We define a planar diffeomorphism $\hat\Psi_{\ep}:\mb R^2\rightarrow\mb R^2$ to be of the form $\hat{\Psi}_{\ep}(\hat x) = \hat{x}-f_{\ep}(d(\hat{x}))\hat{n}(\hat{x})$, where $\hat{n}(\hat{x})$ is the outer unit normal to $\partial\Omega$ passing through $\hat{x}$, $d(\hat{x})$ is the signed distance from $\hat x$ to $\partial\Omega$ and $f_{\ep}$ is a nonnegative smooth function with compact support in $(-\sqrt{\ep},\sqrt{\ep})$ such that $|f_{\ep}'|\leq C\sqrt{\ep}$ and $f_{\ep}(\ep)>\ep$. Then we have $\hat\Psi_{\ep}(\Omega_{\ep})\subset \Omega$ and $\hat\Psi_{\ep}(\hat x)=\hat x$ for $\hat x\in (\R^2\setminus\Omega_{\sqrt{\ep}})\cup\Omega_{-\sqrt{\ep}}$. Further, $\lVert \hat D\hat\Psi_{\ep}-I_{2\times 2}\rVert_{\infty}\leq C\sqrt{\ep}$ and $\lVert \hat D\hat\Psi_{\ep}^{-1}-I_{2\times 2}\rVert_{\infty}\leq C\sqrt{\ep}$, where $I_{2\times 2}$ is the $2\times 2$ identity matrix. Next we take a cut-off function $\eta\in C^{\infty}_c(\R)$ with $0\leq\eta\leq 1$, $\eta\equiv 1$ on $[-1,L+1]$, $\mathrm{supp}(\eta)\subset(-3,L+3)$ and $|\eta'|\leq 1$. We define the diffeomorphism $\Phi_{\ep}: \R^3\rightarrow\R^3$ to be $\Phi_{\ep}(x)=(\hat\Psi_{\ep}(\hat x)\eta(x_3),x_3)$. Because of the properties for $\hat\Psi_{\ep}$ and $\eta$, it is clear that
\begin{equation}\label{eq13}
\Phi_{\ep}(\Omega_{\ep}\times(-\ep,L+\ep))\subset \Omega\times (-\ep,L+\ep)
\end{equation}
and
\begin{equation}\label{eqa11}
\Phi_{\ep}(x)=x \text{ for } x\in \R^3\setminus\Lambda_{\ep},
\end{equation}
where $\Lambda_{\ep}:=(\Omega_{\sqrt{\ep}}\setminus\Omega_{-\sqrt{\ep}})\times(-3,L+3)$. 

Let $(\hat\phi,\vec\psi)\in V\times \check{H}^1_{div}$. Given $\ep>0$ sufficiently small, we can find $s_{1}^{\ep}\in(0,\ep)$ and $s_2^{\ep}\in (L-\ep,L)$ such that $\hat\phi(\cdot,s_j^{\ep})\in L^2(\Omega)$ and $\hat\psi(\cdot,s_j^{\ep})\in \check H^1(\R^2)$ for $j=1,2$. Now we define $P_{\ep}(\hat\phi)$ to be
\begin{equation}\label{eqa10}
P_{\ep}(\hat\phi)(x):=\begin{cases}
\hat\phi(\hat x, x_3+\ep+s_1^{\ep})& \text{for }x_3<-\ep,\\
\hat\phi(\hat x,s_1^{\ep})& \text{for }x_3\in(-\ep,s_1^{\ep}),\\
\hat\phi(x) &\text{for }x_3\in[s_1^{\ep},s_2^{\ep}],\\
\hat\phi(\hat x,s_2^{\ep}) & \text{for }x_3\in(s_2^{\ep},L+\ep),\\
\hat\phi(\hat x, x_3-\ep-L+s_2^{\ep}) & \text{for }x_3>L+\ep,
\end{cases}
\end{equation}
where $\hat \phi$ is extended to be zero outside $D$. Similarly we can define $P_{\ep}(\hat\psi)$ by (\ref{eqa10}) with $\hat\phi$ replaced by $\hat\psi$. By continuity of translation, as $\ep\rightarrow 0^+$, we have $P_\ep(\hat\phi)\rightarrow\hat\phi$ in $L^2(\R^3;\R^2)$ and $(P_\ep(\hat\psi),\psi^3)\rightarrow\vec\psi$ in $\check{H}^1(\R^3;\R^3)$, and thus $(P_\ep(\hat\psi),\psi^3)\rightarrow\vec\psi$ in $L^6(\R^3;\R^3)$ by (\ref{H11}) (here we implicitly choose the representative of $\check{H}^1$ elements which also belong to $L^6$). Given $\sigma>0$, let $\rho_{\sigma}$ denote the standard mollifier, i.e., $\rho_{\sigma}\in C^{\infty}_c(\R^3)$ with $\supp(\rho_{\sigma})\subset B_{\sigma}(0)\subset \R^3$ and $\int_{\R^3}\rho_{\sigma} dx = 1$. Define $\hat{\phi}_{\ep} := \rho_{\sigma(\ep)}\ast\hat{\Phi}_{\ep}^*(P_{\ep}(\hat\phi))$, $\hat{\psi}_{\ep}:=\rho_{\sigma(\ep)}\ast\hat{\Phi}_{\ep}^*(P_{\ep}(\hat\psi))$ and $\psi^3_{\ep}:=\rho_{\sigma(\ep)}\ast\psi^3$, where $0<\sigma(\ep)<\ep$ is sufficiently small depending on the size of $\ep$ and $\hat\Phi_{\ep}^*$ denotes the pullback of $\hat\Phi_{\ep}$, i.e., $\hat\Phi_{\ep}^*(\hat w)(x) = [\hat D\hat\Phi_{\ep}]^T \hat{w}(\Phi_{\ep}(x))$. 

Now we verify that the sequence $\{(\hat{\phi}_{\ep},\vec{\psi}_{\ep})\}\subset C_c^{\infty}(\R^3;\R^2)\times C_0^{\infty}(\R^3;\R^3)$ constructed above satisfies the properties in Lemma \ref{l7}. For any $\delta>0$, we have
\begin{equation}\label{eq15}
\begin{split}
&\int_{\R^3}|\hat{\phi}_{\ep}-\hat{\phi}|^2\, dx\leq C\int_{\R^3} |\hat{\phi}_{\ep}-\hat\Phi_{\ep}^*(P_{\ep}(\hat{\phi}))|^2\,dx\\
&\qd\qd\qd\qd\qd\qd\qd\qd + C\int_{\R^3} |\hat\Phi_{\ep}^*(P_{\ep}(\hat{\phi}))-P_{\ep}(\hat{\phi})|^2dx+C\int_{\R^3}|P_{\ep}(\hat\phi)-\hat\phi|^2dx.
\end{split}
\end{equation}
Note that by (\ref{eqa11}), we have
\begin{equation*} 
\int_{\R^3} |\hat\Phi_{\ep}^*(P_{\ep}(\hat{\phi}))-P_{\ep}(\hat{\phi})|^2dx = \int_{\Lambda_{\ep}} |\hat\Phi_{\ep}^*(P_{\ep}(\hat{\phi}))-P_{\ep}(\hat{\phi})|^2dx.
\end{equation*}
Since $\lVert \hat D\hat\Phi_{\ep}\rVert_{\infty}<C$ for some constant $C$ independent of $\ep$, it follows that $|\hat\Phi_{\ep}^*(P_{\ep}(\hat{\phi}))(x)|\leq C|P_{\ep}(\hat{\phi})(\Phi_{\ep}(x))|$, and hence the above integral on the right hand side converges to zero as $\ep\rightarrow 0$ by the dominated convergence theorem. Therefore, for $\ep$ sufficiently small, we have
\begin{equation*}
C\int_{\R^3} |\hat\Phi_{\ep}^*(P_{\ep}(\hat{\phi}))-P_{\ep}(\hat{\phi})|^2dx <\frac{\delta}{3}
\end{equation*}
and 
\begin{equation*}
C\int_{\R^3}|P_{\ep}(\hat\phi)-\hat\phi|^2dx <\frac{\delta}{3}.
\end{equation*}
Further, choosing $\sigma(\ep)$ sufficiently small depending on $\ep$, we have
\begin{equation*}
C\int_{\R^3} |\hat{\phi}_{\ep}-\Phi_{\ep}^*(\hat{\phi})|^2dx <\frac{\delta}{3}.
\end{equation*}
Hence we conclude from (\ref{eq15}) that $\hat{\phi}_{\ep}|_{D}\rightarrow \hat{\phi}$ in $L^2(D;\R^2)$. The proof of $\vec{\psi}_{\ep}\rightarrow\vec{\psi}$ in $L^6(\R^3;\R^3)$ follows exactly the same lines. Finally, note that 
$$\lVert\nabla\hat\psi_{\ep}\rVert_{L^2(\R^3)}\leq \lVert\nabla\lt(\hat\Phi_{\ep}^*(P_{\ep}(\hat\psi))\rt)\rVert_{L^2(\R^3)}\leq C\lt(\lVert\nabla P_{\ep}(\hat\psi)\rVert_{L^2(\R^3)}+\lVert P_{\ep}(\hat\psi)\rVert_{L^2(\R^3)}\rt)<C$$ 
for constants $C$ independent of $\ep$. Recall that $\psi^3_{\ep}:=\rho_{\sigma(\ep)}\ast\psi^3$ and thus $\lVert\nabla\psi_\ep^3\rVert_{L^2(\R^3)}\leq \lVert\nabla\psi^3\rVert_{L^2(\R^3)}$. Hence, upon extraction of a subsequence (without relabeled), we have $\nabla\vec\psi_{\ep}\rightharpoonup\nabla\vec\psi$, and, in particular, $\Cl\vec\psi_{\ep}\rightharpoonup\Cl\vec\psi$ in $L^2(\R^3)$.

Finally we verify (\ref{eq11}). To this end, we take a test function $\varphi \in C^{\infty}_c(D)$ with $\lVert \varphi\rVert_{\infty}\leq 1$. We denote $\hat\chi:=\hat{\phi}+\hat{\psi}$ and $\hat{\chi}_{\ep}:= \hat{\phi}_{\ep}+\hat{\psi}_{\ep}$. 
We compute
\begin{equation}\label{eq16}
\begin{split}
\int_{D} \lt(\cl\hat\chi_{\ep}\rt)\varphi\, dx &= -\int_{D} \hat{\chi}_{\ep}\cdot \hat\nabla^{\perp}\varphi \,dx\\
&= -\int_{\R^3} \rho_{\sigma(\ep)}\ast\hat{\Phi}_{\ep}^*(P_{\ep}(\hat\chi)) \cdot \hat\nabla^{\perp}\varphi \, dx \\
&=- \int_{\R^3} \hat{\Phi}_{\ep}^*(P_{\ep}(\hat\chi)) \cdot \hat\nabla^{\perp}\varphi_{\sigma(\ep)} \, dx,
\end{split}
\end{equation}
where $\varphi_{\sigma}:=\rho_{\sigma}\ast\varphi$. Note that $\eta'=0$ in $\mathrm{supp}(\varphi_{\sigma})\subset\Omega_{\ep}\times(-\ep,L+\ep)$. Thus, in $\mathrm{supp}(\varphi_{\sigma})$, we have 
$$D\Phi_{\ep} = \lt(\begin{matrix} \partial_1 \Phi_{\ep}^1 & \partial_2 \Phi_{\ep}^1 & 0 \\ \partial_1 \Phi_{\ep}^2 & \partial_2 \Phi_{\ep}^2 & 0 \\ 0 & 0 & 1 \end{matrix}\rt),$$
and $\det \lt[D\Phi_{\ep}^{-1}\rt]=\det \lt[\hat D\hat\Phi_{\ep}^{-1}\rt]=1/\det \lt[\hat D\hat\Phi_{\ep}\rt]$. Direct calculations using the change of variables $y=\Phi_{\ep}(x)$ give
\begin{equation*}
\hat{\Phi}_{\ep}^*(P_{\ep}(\hat\chi)) \cdot \hat\nabla^{\perp}\varphi_{\sigma}= \det \lt[\hat{D}\hat{\Phi}_{\ep}\rt] \,P_{\ep}(\hat\chi)\cdot\hat\nabla^{\perp}_y \varphi_{\sigma}. 
\end{equation*}
It follows that
\begin{equation}\label{eq18}
\begin{split}
\int_{D} \lt(\cl\hat\chi_{\ep}\rt)\varphi\, dx &= -\int_{\mathrm{supp}(\varphi_{\sigma})}\det \lt[\hat{D}\hat{\Phi}_{\ep}\rt] \,P_{\ep}(\hat\chi)\cdot\hat\nabla^{\perp}_y \varphi_{\sigma} \det \lt[D\Phi_{\ep}^{-1}\rt]\,dy\\
&= -\int_{\R^3}P_{\ep}(\hat\chi)\cdot\hat\nabla^{\perp}_y \varphi_{\sigma}\,dy.
\end{split}
\end{equation}
Note that $\supp(\varphi_{\sigma}(\Phi_{\ep}^{-1}(y)))\subset \Phi_{\ep}(\Omega_{\sigma}\times(-\sigma,L+\sigma))\subset \Phi_{\ep}(\Omega_{\ep}\times(-\ep,L+\ep))\overset{(\ref{eq13})}{\subset} D_{\ep}$, where $D_{\ep}:=\Omega\times(-\ep,L+\ep)$. Therefore, $\varphi_{\sigma}\in C^{\infty}_c(D_{\ep})$ and $\lVert \varphi_{\sigma}\rVert_{\infty} \leq \lVert \varphi\rVert_{\infty}\leq 1$. Since (\ref{eq18}) holds for all $\varphi$, we deduce that
\begin{equation}\label{eqa13}
|\cl\hat\chi_{\ep}|(D)\leq |\cl P_{\ep}(\hat\chi)|(D_{\ep}).
\end{equation}
By construction of $P_{\ep}$, we know that $|\cl P_{\ep}(\hat\chi)|(\Omega\times\{s_{\ep}^j\})=0$ for $j=1,2$. Thus, 
\begin{equation}\label{eqa12}
\begin{split}
|\cl P_{\ep}(\hat\chi)|(D_{\ep})&\leq |\cl\hat\chi|(\Omega\times(s_{\ep}^1,s_{\ep}^2))+|\cl P_{\ep}(\hat\chi)|(\Omega\times(-\ep,s_{\ep}^1))\\
&\qd\qd\qd+|\cl P_{\ep}(\hat\chi)|(\Omega\times(s_{\ep}^2,L+\ep))\\
&\leq |\cl\hat\chi|(D)+(\ep+s_{\ep}^1)|\cl \hat\chi(\cdot,s_{\ep}^1)|(\Omega)\\
&\qd\qd\qd+(L+\ep-s_{\ep}^2)|\cl \hat\chi(\cdot,s_{\ep}^2)|(\Omega).
\end{split}
\end{equation}
Putting (\ref{eqa13}) and (\ref{eqa12}) together and letting $\ep\rightarrow 0$, we obtain
\begin{equation}\label{eq19}
\limsup_{\ep\rightarrow 0} |\cl\hat\chi_{\ep}|(D)\leq |\cl\hat\chi|(D). 
\end{equation}
On the other hand, (\ref{eq19}) implies that $\{\cl\hat\chi_{\ep}\}$ has a subsequence that converges weakly* to some finite Radon measure $\mu$. As $\hat\chi_{\ep}\rightarrow \hat\chi$ in $L^2(D;\R^2)$, it is clear that $\mu = \cl\hat\chi$. By lower semicontinuity of the total variation with respect to the weak* convergence (see, e.g., Theorem 1.59 in \cite{AFP}), we obtain
\begin{equation}\label{eq20}
\liminf_{\ep\rightarrow 0} |\cl\hat\chi_{\ep}|(D)\geq |\cl\hat\chi|(D). 
\end{equation}
Putting (\ref{eq19}) and (\ref{eq20}) together we obtain (\ref{eq11}).  
\end{proof}

%
%

\section{First characterization of $H_{c_1}$: proof of Theorem \ref{t3}}\label{section4}

In this section we give the proof of Theorem \ref{t3}, which provides a characterization for triviality of the limiting vorticity measure. The proof adapts that for Theorem 3 in \cite{BJOS2} to account for the highly anisotropic features in our problem. We will need the following lemma.

\begin{lemma}\label{l4}
	If $\vec B\in\mathcal{C}$, then $\supp(\Cl\vec B)\subset \overline D$, and $(\Cl \vec B)|_{D}\in N^{\perp}$, where $N:=\{\vec C\in L^2(D;\R^3):\cl\hat C=0\}$. Conversely, for any $\vec{\psi}\in N^{\perp}$, there exists $\vec{B}_{\vec\psi}\in\mathcal{C}$ such that $\Cl\vec{B}_{\vec{\psi}}=\vec{\psi}\mathbb{1}_{D}$.
\end{lemma}

We need a couple of auxiliary lemmas. The first is an approximation lemma.

\begin{lemma}\label{l10}
The set $C^{\infty}(\overline{D})\cap N$ is dense in $N$ with respect to the $L^2$ norm.
\end{lemma}

\begin{proof}
Given $\vec{\phi}\in N$, we approximate $\vec\phi$ by smooth vector fields in a similar way as in the proof of Lemma \ref{l7}. Namely, for $\ep>0$ sufficiently small, let $\Phi_{\ep}(x):\R^3\rightarrow \R^3$ be the diffeomorphism defined in the proof of Lemma \ref{l7}. In particular, it satisfies the properties (\ref{eq13})-(\ref{eqa11}). Further, let $P_{\ep}(\hat\phi)$ be defined by (\ref{eqa10}), and denote by $\hat\Phi_{\ep}^*(P_{\ep}(\hat{\phi}))$ the pullback of $P_{\ep}(\hat{\phi})$ under $\hat\Phi_{\ep}$. Finally define $\hat{\phi}_{\ep}:=\rho_{\sigma(\ep)}\ast\hat\Phi_{\ep}^*(P_{\ep}(\hat{\phi}))$ and $\phi_{\ep}^3=\rho_{\sigma(\ep)}\ast\phi^3$ in $D$. Then we have $\vec{\phi}_{\ep} \rightarrow \vec{\phi}$ in $L^2(D;\R^3)$ as can be seen from the proof of Lemma \ref{l7}. Given $\varphi\in C^1_{c}(D)$, similar to the calculations in the proof of Lemma \ref{l7}, specifically, the calculations performed in (\ref{eq16})-(\ref{eq18}), we have
	\begin{equation}\label{eq22}
	\int_{D}\cl\hat{\phi}_{\ep}\,\varphi\, dx=-\int_{\R^3}P_{\ep}(\hat\phi)\cdot\hat\nabla^{\perp}_y \varphi_{\sigma}\,dy
	\end{equation} 
	for $y=\Phi_{\ep}(x)$ and $\varphi_{\sigma}:=\rho_{\sigma(\ep)}\ast\varphi$. As in the proof of Lemma \ref{l7}, $\varphi_{\sigma}$ has compact support in $D_{\ep}=\Omega\times(-\ep,L+\ep)$. Since $\cl\hat{\phi}=0$ in $D$, the construction of $P_{\ep}(\hat\phi)$ implies that $\cl P_{\ep}(\hat\phi)=0$ in $D_{\ep}$. We conclude from (\ref{eq22}) that
	$\int_{D}\cl\hat{\phi}_{\ep}\,\varphi\, dx = 0$
	for all $\varphi\in C^{\infty}_c(D)$ and thus $\cl\hat{\phi}_{\ep}=0$ in $D$. In particular, the sequence $\{\vec{\phi}_{\ep}\}\subset C^{\infty}(\overline{D})\cap N$ and converges to $\vec{\phi}$ in $L^2(D;\R^3)$.
\end{proof}

Next we give a characterization of the space $N^{\perp}$.

\begin{lemma}\label{l11}
We have
\begin{equation}\label{eq25}
N^{\perp}=\lt\{\vec{C}\in L^2(D;\R^3): \nabla\cdot\vec{C}=0, C^3=0, \vec{C}\cdot\vec{n}=0 \text{ in } H^{-\frac{1}{2}}(\partial D)\rt\}.
\end{equation}
\end{lemma}

\begin{proof}
We first show that, given $\vec{\phi}\in N^{\perp}$, it belongs to the space on the right hand side in (\ref{eq25}). To this end, first note that any $\vec{\chi}=(0,0,\chi^3)$ with $\chi^3\in L^2(D)$ belongs to $N$. Therefore, using arguments similar to those in the proof of Corollary \ref{c3}, it follows immediately that $\phi^3=0$.

To see that $\nabla\cdot\vec{\phi}=0$ in the sense of distributions, take any test function $\varphi \in C^{\infty}_c (D)$. It is clear that $\nabla\varphi\in N$. Therefore we have
\begin{equation*}
\int_{D} \vec{\phi}\cdot \nabla\varphi \,dx = 0 \qd\text{ for all } \varphi\in C^{\infty}_c(D),
\end{equation*}
which reads as $\nabla\cdot\vec{\phi}=0$ in the sense of distributions. Now for any $\varphi\in H^1(D)$, by Green's formula (see, e.g., equation (2.17) in \cite{GR}), we have
\begin{equation*}
\int_{D}\vec\phi\cdot\nabla\varphi\,dx + \int_{D}\lt(\nabla\cdot\vec\phi\rt)\varphi\,dx=\left\langle\vec\phi\cdot\vec n,\varphi \right\rangle_{\partial D}, 
\end{equation*}
where the right hand side in the above equation denotes the duality between $H^{-\frac{1}{2}}$ and $H^{\frac{1}{2}}$ on $\partial D$. It is clear that $\nabla\varphi\in N$ and thus $\int_{D}\vec\phi\cdot\nabla\varphi\,dx=0$. It follows that $\left\langle\vec\phi\cdot\vec n,\varphi \right\rangle_{\partial D}=0$ for all $\varphi|\in H^{1}(D)$. As the trace operator is onto $H^{\frac{1}{2}}(\partial D)$, we deduce that $\vec\phi\cdot\vec n=0$ in $H^{-\frac{1}{2}}(\partial D)$. This completes the proof of the inclusion 
\begin{equation*}
N^{\perp}\subset\lt\{\vec{C}\in L^2(D;\R^3): \nabla\cdot\vec{C}=0, C^3=0, \vec{C}\cdot\vec{n}=0 \text{ in } H^{-\frac{1}{2}}(\partial D)\rt\}.
\end{equation*}

To see the reverse inclusion, let $\vec{\phi}\in L^2(D;\R^3)$ satisfy $\nabla\cdot\vec{\phi}=0$, $\phi^3=0$, and $\vec{\phi}\cdot\vec{n}=0$ in $H^{-\frac{1}{2}}(\partial D)$ and $\vec{\chi}\in N$. Let $\{\vec{\phi}_{\ep}\}\subset C^{\infty}(\overline{D})\cap \{\vec{C}\in L^2(D;\R^3): \nabla\cdot\vec{C}=0, C^3=0, \vec{C}\cdot\vec{n}=0 \text{ in } H^{-\frac{1}{2}}(\partial D)\}$  and  $\{\vec\chi_{\ep}\}\subset C^{\infty}(\overline{D})\cap N$ be sequences that converge to $\vec{\phi}$ and $\vec{\chi}$ in $L^2(D)$, respectively. For all $x_3\in(0,L)$, as $\hat\nabla\cdot\hat{\phi}_{\ep}=0$, $\hat\phi_{\ep}\cdot\hat{n}=0$ on $\partial\Omega$, and $\cl\hat{\chi}_{\ep}=0$, by standard Hodge decomposition, we have that
\begin{equation*}
\int_{\Omega} \hat\phi_{\ep}(\hat x,x_3)\cdot \hat\chi_{\ep}(\hat x, x_3)\, d\hat x = 0.
\end{equation*}
It follows that
\begin{equation*}
\int_{D} \vec{\phi}\cdot \vec{\chi}\,dx = \lim_{\ep\rightarrow 0} \int_{D} \vec{\phi}_{\ep}\cdot \vec{\chi}_{\ep}\,dx = \lim_{\ep\rightarrow 0} \int_{0}^{L}\int_{\Omega}\hat\phi_{\ep}(\hat x,x_3)\cdot \hat\chi_{\ep}(\hat x, x_3)\, d\hat x\, dx_3 = 0
\end{equation*}
and hence $\vec{\phi}\in N^{\perp}$. 
\end{proof}

\begin{proof}[Proof of Lemma \ref{l4}]
	Since $N$ is a closed subspace of $L^2(D;\R^3)$, we have the decomposition $L^2(D;\R^3)=N\oplus N^{\perp}$. We first show that
	\begin{equation}\label{l4.1}
	\int_{\R^3}\lt(\Cl\vec B\rt)\cdot\vec{\phi}\,dx = 0 \text{ for all } \vec{\phi}\in L^2(\R^3;\R^3) \text{ s.t. } \cl\hat{\phi}=0 \text{ in } D.
	\end{equation}
	Given $\vec{\phi}$ as above, let $\vec{\phi}_{\ep}$ be as in the proof of Lemma \ref{l10}. Then we have $\vec{\phi}_{\ep}\rightarrow \vec{\phi}$ in $L^2(\R^3;\R^3)$ and $\cl\hat{\phi}_{\ep}=0$ in $D$. Hence by definition of the space $\mathcal{C}$ we have
	\begin{equation*}
	\int_{\R^3}\lt(\Cl\vec B\rt)\cdot\vec{\phi}\,dx=\lim_{\ep\rightarrow 0}\int_{\R^3}\lt(\Cl\vec B\rt)\cdot\vec{\phi}_{\ep}\, dx=0.
	\end{equation*}
	This proves \eqref{l4.1}. 
	
	To see that $\supp(\Cl\vec B)\subset \overline D$, we take any nonnegative $\chi\in C^{\infty}_c(\R^3\setminus\overline{D})$. Then $\chi\vec{B} \in L^2(\R^3;\R^3)$ and $\cl(\chi\hat{B})=0$ in $D$. It follows from (\ref{l4.1}) that
	\begin{equation*}
	\int_{\R^3} \chi \lt|\Cl\vec{B}\rt|^2 \, dx =0
	\end{equation*}
	for all nonnegative $\chi\in C^{\infty}_c(\R^3\setminus\overline{D})$, and hence we have $\supp(\Cl\vec{B})\subset \overline{D}$. Now for any $\vec{\phi}\in N$, it follows from (\ref{l4.1}) that
	\begin{equation*}
	0=\int_{\R^3} \lt(\Cl\vec{B}\rt) \cdot \vec{\phi}\,\mathbb{1}_{D} \, dx = \int_{D} \lt(\Cl\vec{B}\rt) \cdot \vec{\phi} \, dx,
	\end{equation*}
	and hence $(\Cl\vec{B})|_{D}\in N^{\perp}$.
	
	Given $\vec{\psi}\in N^{\perp}$, let $\vec{\chi}:= \Gamma_3\ast\lt(\vec{\psi}\mathbb{1}_D\rt)$ where $\Gamma_3$ is the fundamental solution for the Laplacian in $\R^3$ and define $\vec{B}_{\vec{\psi}}:= \Cl\vec{\chi}$. By Theorem 9.9 in \cite{GT}, $\vec\chi\in H^2(\R^3;\R^3)$ and thus $\vec{B}_{\vec{\psi}}\in H^1(\R^3;\R^3)\cap \Cl H^1(\R^3;\R^3)$.  By Lemma \ref{l11}, as $\nabla\cdot\vec\psi=0$ in $D$ and $\vec\psi\cdot\vec n=0$ on $\partial D$, we have $\nabla\cdot\lt(\vec\psi\mathbb{1}_{D}\rt)=0$ in $\R^3$. Thus $\nabla\cdot\vec{\chi}=0$, and it follows that $\Cl\vec{B}_{\vec{\psi}} = -\Delta\vec{\chi}=\vec{\psi}\mathbb{1}_D$ and $(\Cl\vec{B}_{\vec{\psi}})|_D = \vec{\psi} \in N^{\perp}$. For any $\vec\phi\in H^1(\R^3;\R^3)$ with $\cl\hat\phi=0$ in $D$, by virtue of the proof of Lemmas \ref{l7} and \ref{l10}, we can find a sequence $\{\vec\phi_{\ep}\}\subset C^{\infty}_{0}(\R^3;\R^3)$ such that $\vec\phi_{\ep}\rightharpoonup\vec\phi$ in $H^1(\R^3;\R^3)$ and $\cl\hat\phi_{\ep}=0$ in $D$. In particular, $\vec\phi_{\ep}|_{D}\in N$. Hence we have
	\begin{equation*}
	0=\int_{D}(\Cl\vec B_{\vec\psi})\cdot\vec\phi_{\ep}\,dx = \int_{\R^3}(\Cl\vec B_{\vec\psi})\cdot\vec\phi_{\ep}\,dx=\int_{\R^3}\vec B_{\vec\psi}\cdot(\Cl\vec\phi_{\ep})\,dx
	\end{equation*}
	for all $\ep$, and thus
	\begin{equation*}
	\int_{\R^3}\vec B_{\vec\psi}\cdot\lt(\Cl\vec\phi\rt)dx = \lim_{\ep\rightarrow 0}\int_{\R^3}\vec B_{\vec\psi}\cdot\lt(\Cl\vec\phi_{\ep}\rt)dx=0
	\end{equation*}
	for all $\vec\phi\in H^1(\R^3;\R^3)$ with $\cl\hat\phi=0$ in $D$. This shows that $\vec B_{\vec\psi}\in\mathcal{C}$.
\end{proof}

Now we give the proof of Theorem \ref{t3}.

\begin{proof}[Proof of Theorem \ref{t3}]
	Let us first assume $\vec B_*=\vec B_0$ and we need to show that $\cl \hat v_0=0$. Note that $\nabla\cdot\vec{B}_0 = 0$ a.e. as $\vec{B}_0=\Cl(\vec{A}_0-\vec{a})$, and therefore $\Cl\Cl\vec{B}_0 = -\Delta\vec{B}_0$ in the weak sense. By \eqref{t1.2}, we have
	\begin{equation*}
	\cl \hat v_0=-\Delta B_0^3+B_0^3+1 \text{ in } D
	\end{equation*}
	in the sense of distributions. So it suffices to show that 
	\begin{equation}\label{eq50}
	-\Delta B_*^3+B_*^3+1=0 \text{ in } D.
	\end{equation} 
	By direct variation of $\mathcal{E}_0$ in the set $\mathcal{C}$, we obtain
	\begin{equation}\label{t31}
	\int_{\R^3}\vec{B}_*\cdot\vec{B}\,dx+\int_{D}\lt(\Cl\vec{B}_*+\vec{a}\rt)\cdot\lt(\Cl\vec{B}\rt)dx=0 \text{ for all } \vec{B}\in\mathcal{C}.
	\end{equation}
	Since $\nabla\cdot\vec{B}_*=0$, there exists $\vec\psi\in H^2(\R^3;\R^3)$ such that $\Cl\vec{\psi}=\vec{B}_*$. One can choose $\vec\psi=(-\Delta)^{-1}(\Cl\vec B_*)$ and standard elliptic regularity implies $\vec\psi\in H^2(\R^3;\R^3)$. It follows from \eqref{t31} that
	\begin{equation*}
	\begin{split}
	0&=\int_{\R^3}\lt(\Cl\vec{\psi}\rt)\cdot\vec{B}\,dx+\int_{D}\lt(\Cl\vec{B}_*+\vec{a}\rt)\cdot\lt(\Cl\vec{B}\rt)dx\\
	&=\int_{\R^3}\vec{\psi}\cdot\lt(\Cl\vec{B}\rt)dx+\int_{D}\lt(\Cl\vec{B}_*+\vec{a}\rt)\cdot\lt(\Cl\vec{B}\rt)dx\\
	&=\int_{D}\lt(\vec{\psi}+\Cl\vec{B}_*+\vec{a}\rt)\cdot\lt(\Cl\vec{B}\rt)dx
	\end{split}
	\end{equation*}
	for all $\vec{B}\in\mathcal{C}$. We deduce from Lemma \ref{l4} that $\vec{\psi}+\Cl\vec{B}_*+\vec{a}\in (N^{\perp})^{\perp}=N$, and therefore the $x_3$-component of $\Cl\lt(\vec{\psi}+\Cl\vec{B}_*+\vec{a}\rt)$ equals zero, which is exactly (\ref{eq50}) as desired.
	
	Next assume $\cl\hat v_0=0$ and we show $\vec{B}_*=\vec{B}_0$. It suffices to show
	\begin{equation}\label{t32}
	\mathcal{E}_0(\vec{B}_*)=\mathcal{E}_0(\vec{B}_0).
	\end{equation}
	By Lemma \ref{l4}, we have
	\begin{equation}\label{t33}
	\int_{D}\vec{v}_0\cdot\lt(\Cl\vec{B}\rt)dx=0 \text{ for all } \vec{B}\in\mathcal{C}.
	\end{equation}
	Plugging $\vec{B}_0$ in \eqref{t33} and using \eqref{t1.2}, we have
	\begin{equation}\label{eqa14}
	\begin{split}
	0&=\int_{D}\lt(\Cl\vec{B}_0+\vec{A}_0\rt)\cdot\lt(\Cl\vec{B}_0\rt)dx\\
	&=\int_{\R^3}\lt(\Cl\vec{B}_0+(\vec{A}_0-\vec a)+\vec a\rt)\cdot\lt(\Cl\vec{B}_0\rt)dx\\
	&=\int_{D}\lt|\Cl\vec{B}_0\rt|^2dx+\int_{\R^3}\lt(\Cl(\vec{A}_0-\vec a)\rt)\cdot\vec{B}_0\,dx+\int_{\R^3}\vec a\cdot \lt(\Cl\vec B_0\rt)\,dx\\
	&=\int_{D}\lt|\Cl\vec{B}_0\rt|^2dx+\int_{\R^3}\lt|\vec{B}_0\rt|^2dx+\int_{D}\vec{a}\cdot\lt(\Cl\vec{B}_0\rt)\,dx,
	\end{split}
	\end{equation}
	where the integration by parts can be easily justified by approximation by smooth functions. Using the above, one can rewrite
	\begin{equation}\label{t34}
	\begin{split}
	\mathcal{E}_0(\vec{B}_0)&=\frac{1}{2}\int_{\mb R^3}\lt|\vec B_0\rt|^2dx+\frac{1}{2}\int_{D}\lt|\nabla\times\vec B_0+\vec a\rt|^2\,dx\\
	&=\frac{1}{2}\int_{D}\lt(\vec{a}\cdot\lt(\Cl\vec{B}_0\rt)+|\vec a|^2\rt)dx.
	\end{split}
	\end{equation}
	Next, using $\vec{B}_*$ as a test function in \eqref{t31} and following similar calculations as above, we can rewrite
	\begin{equation}\label{t35}
	\mathcal{E}_0(\vec{B}_*)=\frac{1}{2}\int_{D}\lt(\vec{a}\cdot\lt(\Cl\vec{B}_*\rt)+|\vec a|^2\rt)dx.
	\end{equation}
	Using $\vec{B}_0$ as a test function in \eqref{t31}, we obtain
	\begin{equation}\label{t36}
	0=\int_{\R^3}\vec{B}_*\cdot\vec{B}_0\,dx+\int_{D}\lt(\Cl\vec{B}_*+\vec{a}\rt)\cdot\lt(\Cl\vec{B}_0\rt)dx.
	\end{equation}
	Plugging $\vec{B}_*$ in \eqref{t33} and using \eqref{t1.2} and exactly the same lines as in (\ref{eqa14}), we obtain
	\begin{equation}\label{t37}
	0=\int_{D}\lt(\Cl\vec B_0\rt)\cdot\lt(\Cl \vec B_*\rt)dx+\int_{\R^3}\vec{B}_0\cdot\vec{B}_*\,dx+\int_{D}\vec a\cdot\lt(\Cl\vec{B}_*\rt)dx.
	\end{equation}
	Comparing \eqref{t36} with \eqref{t37}, we see that
	\begin{equation*}
	\int_{D} \vec{a}\cdot\lt(\Cl\vec{B}_0\rt)dx = \int_{D} \vec{a}\cdot\lt(\Cl\vec{B}_*\rt)dx.
	\end{equation*}
	This together with (\ref{t34}) and (\ref{t35}) gives (\ref{t32}) and hence $\vec{B}_* = \vec{B}_0$. This completes the proof of Theorem \ref{t3}.
\end{proof}

%
%

\section{More explicit characterization of $H_{c_1}$: proof of Theorem \ref{t4}}\label{section5}

The proof of Theorem \ref{t4} requires some preparation. Recall from Theorem \ref{t3} that $\vec{B}_*$ denotes the unique minimizer of the energy functional $\mathcal{E}_0$ in the space $\mathcal{C}$. From Lemma \ref{l012}, there exists a unique $\vec{A}_*\in K_0$ such that $\Cl\lt(\vec{A}_*-\vec{a}\rt) = \vec{B}_*$. Then we have the following key lemma.

\begin{lemma}\label{l013}
	Let $(\hat v_0,\vec A_0)\in V\times K_0$ be a minimizer of $\mathcal{G}_{h_0}$ and denote $\vec{B}_0=\Cl(\vec A_0-\vec a)$. Let $\vec{B}_*$ be the unique minimizer of $\mathcal{E}_0$ in the space $\mathcal{C}$ and $\vec{A}_*\in K_0$ be the unique element satisfying $\Cl\lt(\vec{A}_*-\vec{a}\rt) = \vec{B}_*$. Further let $\hat{v}_* \in V$ be the unique minimizer of the functional
	\begin{equation}\label{ep11}
	F(\hat v;\vec A_{*}):= \frac{1}{2}\int_{D} \lt|\hat v - \hat A_*\rt|^2\, dx + \frac{1}{2h_0}\lt|\cl\hat v\rt|(D).
	\end{equation}
	Then we have $\cl \hat v_0 = 0$ if and only if $\cl \hat v_* = 0$.
\end{lemma}	

\begin{proof}
	First assume that $\cl \hat v_0 = 0$. By Theorem \ref{t3}, we have $\vec B_* = \vec B_0$. As $\vec A_0\in K_0$ and $\Cl(\vec A_0-\vec a)=\vec B_0=\vec B_*=\Cl(\vec A_*-\vec a)$, by Lemma \ref{l012} we have $\vec A_0=\vec A_*$. If $\hat v_*\ne\hat v_0$, then $F(\hat v_*;\vec A_*)<F(\hat v_0;\vec A_0)$ and it would follow that $\mathcal{G}_{h_0}(\hat v_*,\vec A_0)<\mathcal{G}_{h_0}(\hat v_0,\vec A_0)$, which is a contradiction as $(\hat v_0,\vec A_0)$ minimizes $\mathcal{G}_{h_0}$. This implies that $\hat v_* = \hat v_0$ and hence $\cl \hat v_*=0$. 
	
	Next assume that $\cl \hat v_*=0$. We perform the convex duality arguments for $F(\hat v;\vec A)$ as in the proof of Theorem \ref{t1} with the Hilbert space $H=L^2(D;\R^2)$. More precisely, let $\hat\xi=\hat v-\hat A\mathbb{1}_D$. Then we rewrite $F$ as
	\begin{equation*}
	F(\hat v;\vec A)=\mathcal{F}(\hat\xi)=\frac{1}{2}\lVert\hat\xi\rVert_{L^2(D)}^2+\Phi(\hat\xi),
	\end{equation*}
	where $\Phi(\hat\xi)=\frac{1}{2h_0}|\cl(\hat\xi+\hat A)|(D)$. Using (\ref{eqp1}), we compute
	\begin{equation*}
	\begin{split}
	\Phi^*(\hat\xi)&=\sup_{\hat\phi\in H}\lt(\int_{D}\hat\xi\cdot\hat\phi\, dx - \frac{1}{2h_0}\lt|\cl(\hat\phi+\hat A)\rt|(D)\rt)\\
	&=\sup_{\hat\phi\in H}\lt(\int_{D}\hat\xi\cdot\lt(\hat\phi+\hat A\rt) dx - \frac{1}{2h_0}\lt|\cl(\hat\phi+\hat A)\rt|(D)\rt)-\int_{D}\hat\xi\cdot\hat A\,dx\\
	&=\sup_{\hat\phi\in H}\lt(\int_{D}\hat\xi\cdot\hat\phi\, dx - \frac{1}{2h_0}\lt|\cl\hat\phi\rt|(D)\rt)-\int_{D}\hat\xi\cdot\hat A\,dx=-\int_{D}\hat\xi\cdot\hat A\,dx
	\end{split}
	\end{equation*}
	provided
	\begin{equation}\label{eqs22}
	\int_{D}\hat\xi\cdot\hat\phi\, dx\leq\frac{1}{2h_0}|\cl\hat\phi|(D) \text{ for all } \hat\phi\in L^2(D;\R^2),
	\end{equation}
	where $|\cl\hat\phi|(D)$ is understood to equal $+\infty$ if $\hat\phi\notin V$, and $\Phi^*(\hat\xi)=+\infty$ if (\ref{eqs22}) fails to hold true. By Lemma \ref{l2}, we have $\hat\xi$ minimizes $\mathcal{F}$ if and only if $\hat\xi$ minimizes $\mathcal{F}^*$, where
	\begin{equation*}
	\mathcal{F}^*(\hat\xi):=\frac{1}{2}\int_{D}\lt|\hat\xi\rt|^2dx+\int_{D}\hat\xi\cdot \hat A\,dx
	\end{equation*}
	provided (\ref{eqs22}) holds true. As $\hat v_*$ is the minimizer of $F$, we know that $\hat\xi_*:=\hat v_*-\hat A_*\mathbb{1}_D$ satisfies (\ref{eqs22}) and it follows that $\vec\xi_*=(\hat\xi_*,0)\in N^{\perp}$, where the space $N$ is defined in Lemma \ref{l4}. Using Lemma \ref{l4}, there exists $\vec B_{\vec\xi_*}\in \mathcal{C}$ such that $\Cl\vec B_{\vec\xi_*} = \vec\xi_*\mathbb{1}_D$. 
	
	We claim that $\Cl\vec B_{*} = \vec\xi_*\mathbb{1}_D$. Indeed, using (\ref{t31}) with $\vec B = \vec B_*$ and $\vec B = \vec B_{\vec\xi_*}$, we obtain
	\begin{equation}\label{eq41}
	\int_{\R^3}\lt|\vec{B}_*\rt|^2\,dx+\int_{D}\lt|\Cl\vec{B}_*\rt|^2 \,dx+\int_{D}\vec{a}\cdot\lt(\Cl\vec{B}_*\rt)dx=0
	\end{equation}
	and 
	\begin{equation}\label{eq42}
	\int_{\R^3}\vec{B}_*\cdot\vec{B}_{\vec\xi_*}\,dx+\int_{D}\lt(\Cl\vec{B}_*\rt)\cdot\lt(\Cl\vec{B}_{\vec\xi_*}\rt)\,dx + \int_{D}\vec a \cdot\lt(\Cl\vec B_{\vec\xi_*}\rt)\,dx=0.
	\end{equation}
	On the other hand, as $\cl \hat v_*=0$, it follows from Lemma \ref{l4} that
	\begin{equation}\label{eq43}
	\int_{D}\vec{v}_*\cdot\lt(\Cl\vec{B}\rt)dx=0 \text{ for all } \vec{B}\in\mathcal{C}.
	\end{equation}
	Note that $\vec v_* =\Cl\vec B_{\vec\xi_*}  + \vec A_*$ and $\Cl\lt(\vec{A}_*-\vec{a}\rt) = \vec{B}_*$. Using $\vec B=\vec B_*$ in (\ref{eq43}), we obtain
	\begin{equation}\label{eq44}
	\begin{split}
	0&=\int_{D}\lt(\Cl\vec{B}_{\vec\xi_*}+\vec{A}_*\rt)\cdot\lt(\Cl\vec{B}_*\rt)dx\\
	&=\int_{\R^3}\lt(\Cl\vec{B}_{\vec\xi_*}+(\vec{A}_*-\vec a)+\vec a\rt)\cdot\lt(\Cl\vec{B}_*\rt)dx\\
	&=\int_{D}\lt(\Cl\vec{B}_{\vec\xi_*}\rt)\cdot\lt(\Cl\vec{B}_{*}\rt)dx+\int_{\R^3}\lt|\vec{B}_*\rt|^2dx+\int_{D}\vec{a}\cdot\lt(\Cl\vec{B}_*\rt)\,dx.
	\end{split}
	\end{equation}
	Similar calculations using $\vec B=\vec B_{\vec\xi_*}$ in (\ref{eq43}) give
	\begin{equation}\label{eq45}
	\int_{D}\lt|\Cl\vec{B}_{\vec\xi_*}\rt|^2+\int_{\R^3}\vec{B}_*\cdot \vec B_{\vec\xi_*}\,dx+\int_{D}\vec{a}\cdot\lt(\Cl\vec{B}_{\vec\xi_*}\rt)\,dx = 0. 
	\end{equation}
	Comparing (\ref{eq41}) with (\ref{eq44}) we observe that
	\begin{equation}\label{eq46}
	\int_{D}\lt|\Cl\vec{B}_*\rt|^2 \,dx = \int_{D}\lt(\Cl\vec{B}_{\vec\xi_*}\rt)\cdot\lt(\Cl\vec{B}_{*}\rt)dx.
	\end{equation}
	Next from (\ref{eq42}) and (\ref{eq45}) we have that
	\begin{equation}\label{eq47}
	\int_{D}\lt(\Cl\vec{B}_{\vec\xi_*}\rt)\cdot\lt(\Cl\vec{B}_{*}\rt)dx = \int_{D}\lt|\Cl\vec{B}_{\vec\xi_*}\rt|^2 \,dx.
	\end{equation}
	It follows from (\ref{eq46}) and (\ref{eq47}) that
	\begin{equation*}
	\int_{D} \lt| \Cl\vec{B}_{\vec\xi_*} - \Cl\vec{B}_{*}\rt|^2 dx = 0,
	\end{equation*}
	and hence $\vec\xi_*\mathbb{1}_D = \Cl\vec B_{\vec\xi_*} = \Cl\vec{B}_{*}$. 
	
	Now given $\vec\phi \in H^1(\R^3;\R^3)$ with $\int_{D}\lt|\cl\hat\phi\rt|\,dx\leq 1$, as $\hat\xi_*$ satisfies (\ref{eqs22}) and $\Cl\vec{B}_{*}=\vec\xi_*\mathbb{1}_D$, we have
	\begin{equation*}
	\begin{split}
	\int_{\R^3}\vec B_*\cdot \lt(\Cl\vec\phi\rt)\,dx &= \int_{\R^3}\lt(\Cl\vec B_*\rt)\cdot \vec\phi\,dx\\
	&=\int_{D}\lt(\Cl\vec B_*\rt)\cdot \vec\phi\,dx \\
	&=\int_{D}\hat\xi_* \cdot \hat\phi\,dx \leq \frac{1}{2h_0} \int_{D}\lt|\cl\hat \phi\rt|\,dx \leq \frac{1}{2h_0},
	\end{split}
	\end{equation*}
	where the integration by parts can be easily justified by approximation. Hence, by definition of the $\lVert\cdot\rVert_*$ norm in (\ref{eq48}), we have $\lVert \vec B_*\rVert_* \leq \frac{1}{2h_0}$. It follows from Theorem \ref{t3} that $\cl\hat v_0 = 0$ and this completes the proof of the lemma.
\end{proof}

We need an additional technical lemma.

\begin{lemma}\label{l13}
	For any $\hat v\in V$, we have
	\begin{equation}\label{eq305}
	|\cl \hat v|(D) = \int_{0}^{L} |\cl \hat v(\cdot,x_3)|(\Omega)\, dx_3.
	\end{equation}
\end{lemma}

\begin{proof}
	Given $\hat v\in V$, let $\{\hat v_k\}_k\subset C^{\infty}(\overline{D};\R^2)$ be a sequence such that $\hat v_k\rightarrow \hat v$ in $L^2(D;\R^2)$ and $|\cl\hat v_k|(D) \rightarrow |\cl\hat v|(D)$. Such a sequence exists from the proof of Lemma \ref{l7} (see also Proposition 2.3 in \cite{Pe}). By Fubini's theorem, for a.e. $x_3\in(0,L)$, the function $\hat v(\cdot,x_3)$ is measurable and belongs to $L^2(\Omega;\R^2)$. For such $x_3$, we define
	\begin{equation}\label{eqb1}
	f(x_3):= |\cl \hat v(\cdot,x_3)|(\Omega) = \sup_{\phi\in C^1_{c}(\Omega),\sup|\phi|\leq 1} \int_{\Omega} \hat v(\hat x,x_3)\cdot\hat\nabla^{\perp}\phi(\hat x)\,d\hat x
	\end{equation}
	and
	\begin{equation}\label{eqb2}
	g(x_3):= \liminf_{k\rightarrow \infty} \int_{\Omega}|\cl\hat v_{k}(\hat x,x_3)|\,d\hat x.
	\end{equation}
	As $\hat v_k\rightarrow \hat v$ in $L^2(D;\R^2)$, we have
	\begin{equation*}
	0=\lim_{k\rightarrow \infty} \int_{D}|\hat v_k-\hat v|^2 \,dx = \lim_{k\rightarrow\infty} \int_{0}^{L}\int_{\Omega} |\hat v_k-\hat v|^2\,d\hat x dx_3.
	\end{equation*}
	Therefore, denoting $h_k(x_3):= \int_{\Omega} |\hat v_k(\hat x,x_3)-\hat v(\hat x,x_3)|^2\,d\hat x$, it follows that $h_{k}\rightarrow 0$ in $L^1((0,L))$, and up on extraction of a subsequence (not relabeled), we have $h_{k}\rightarrow 0$ a.e. in $(0,L)$.  This reads as $\hat v_k(\cdot,x_3)\rightarrow \hat v(\cdot,x_3)$ in $L^2(\Omega)$ for a.e. $x_3\in (0,L)$, and thus $\cl \hat v_k$ converges to $\cl\hat v$ weakly* as measures. By  lower semicontinuity of the total variation measure with respect to the weak* convergence, we have
	\begin{equation}\label{eq301}
	f(x_3)\overset{(\ref{eqb1})}{=}|\cl\hat v(\cdot,x_3)|(\Omega) \leq \liminf_{k\rightarrow \infty} |\cl\hat v_k(\cdot,x_3)|(\Omega) \overset{(\ref{eqb2})}{=} g(x_3) \text{ for a.e. } x_3\in(0,L). 
	\end{equation}
	
	Note that since $\hat v_k\in C^{\infty}(\overline{D};\R^2)$, it is clear that $|\cl v_k(\cdot,x_3)|(\Omega)$ is a measurable function of $x_3$ and so is $g(x_3)$. For any $\phi\in C^1_{c}(D)$ with $\sup|\phi|\leq 1$, it follows from (\ref{eqb1}) and (\ref{eq301}) that
	\begin{equation*}
	\int_{D} \hat v(x)\cdot\nabla^{\perp}\phi(x)\, dx = \int_{0}^{L}\int_{\Omega}\hat v(\hat x,x_3)\cdot\nabla^{\perp}\phi(\hat x,x_3) \,d\hat x\, dx_3 \leq \int_{0}^{L} g(x_3)\, dx_3. 
	\end{equation*}
	This implies that
	\begin{equation}\label{eq303}
	|\cl\hat v|(D) \leq \int_{0}^{L} g(x_3)\, dx_3.
	\end{equation}
	On the other hand, using (\ref{eqb2}) and Fatou's lemma, we have
	\begin{equation}\label{eq304}
	\begin{split}
	\int_{0}^{L} g(x_3)\, dx_3 &= \int_{0}^{L}\liminf_{k\rightarrow \infty} |\cl \hat v_k(\cdot,x_3)|(\Omega)\, dx_3\\
	&\leq \lim_{k\rightarrow \infty} \int_{0}^{L}|\cl \hat v_k(\cdot,x_3)|(\Omega)\, dx_3\\
	& = \lim_{k\rightarrow \infty} |\cl\hat v_k|(D) = |\cl\hat v|(D). 
	\end{split}
	\end{equation}
	Combining (\ref{eq303}) with (\ref{eq304}) we obtain
	\begin{equation}\label{eqb3}
	|\cl\hat v|(D) = \int_{0}^{L} g(x_3)\, dx_3.
	\end{equation}
	
	Finally we show that $f(x_3)=g(x_3)$ a.e. in $(0,L)$. This together with (\ref{eqb3}) implies that $f(x_3)$ is measurable and (\ref{eq305}) holds true. Let us denote by $m^*$ the outer measure on $\R$. Define $U:=\{x_3\in(0,L): g(x_3)>f(x_3)\}$ and $U_j:=\{x_3\in(0,L): g(x_3)-f(x_3)>\frac{1}{j}\}$. It follows from (\ref{eq301}) that
	\begin{equation}\label{eqb4}
	U=\lt(\cup_{j} U_j\rt) \bigcup Z
	\end{equation}
	for some $Z$ with $|Z|=0$. Now we claim that $m^*(U_j)=0$ for all $j$. Using (\ref{eqb3}), for all $l$, there exists $\phi_l\in C^1_{c}(D)$ with $\sup|\phi_l|\leq 1$ such that
	\begin{equation*}
	\frac{1}{l}>\int_{0}^{L} g(x_3)\, dx_3-\int_{D} \hat v\cdot\hat\nabla^{\perp}\phi_l\, dx = \int_{0}^L\lt(g(x_3)-\int_{\Omega}\hat v\cdot\hat\nabla^{\perp}\phi_l\,d\hat x\rt) dx_3. 
	\end{equation*}
	Now we denote by $U_j^l:=\{x_3\in(0,L): g(x_3)-\int_{\Omega}\hat v\cdot\hat\nabla^{\perp}\phi_l\,d\hat x >\frac{1}{j}\}$. From (\ref{eqb1}), it is clear that $U_j\subset U_j^{l}$ for all $l$ and hence
	\begin{equation*}
	m^*(U_j) \leq |U_j^{l}| \text{ for all } l.
	\end{equation*}
	By Chebyshev's inequality, we have $|U_j^{l}|\leq \frac{j}{l}\rightarrow 0$ as $l\rightarrow \infty$ and hence $m^*(U_j)=0$ for all $j$. We deduce from (\ref{eqb4}) that $|U|=0$ and hence $f(x_3)=g(x_3)$ a.e. in $(0,L)$.
\end{proof}

\begin{proof}[Proof of Theorem \ref{t4}]
	Recall that $\vec A_*\in K_0$ satisfies $\Cl\lt(\vec{A}_*-\vec a\rt) = \vec B_*$ and $\nabla\cdot\lt(\vec{A}_*-\vec a\rt)=0$. It follows that $-\Delta \vec A_* = \Cl\Cl\vec A_* = \Cl \vec B_* \in L^2(\R^3;\R^3)$. By standard elliptic regularity and the Sobolev embedding theorem, we have $\vec A_* \in H^2_{loc}(\R^3;\R^3)\hookrightarrow C^{0,\alpha}_{loc}(\R^3;\R^3)$ for some $\alpha<1$. Further, as $B_*^3$ satisfies (\ref{eq50}), it follows from standard elliptic regularity that $B_*^3\in C^{\infty}(D)$. Let $\hat v_*$ be as in Lemma \ref{l013}. We first show that, for a.e. $x_3\in(0,L)$, $v_{*}(\cdot,x_3)$ minimizes (noting that $\hat A_*(\cdot,x_3)$ exists in the classical sense)
	\begin{equation*}
	F_{x_3}(\hat v):=\frac{1}{2}\lVert \hat v-\hat A_*(\cdot,x_3) \rVert_{L^2(\Omega)}^2 + \frac{1}{2h_0}|\cl\hat v|(\Omega).
	\end{equation*}
	
	We denote by $\hat v_{x_3}$ the unique minimizer of $F_{x_3}$ in $L^2(\Omega)$. By exactly the same convex duality arguments as in the proof of Lemma \ref{l013}, we obtain that $\hat\xi_{x_3}:=\hat v_{x_3}-\hat A_*(\cdot,x_3)$ minimizes 
	\begin{equation*}
	\mathcal{F}_{x_3}(\hat\xi):=\frac{1}{2}\int_{\Omega}|\hat\xi|^2d\hat x+\frac{1}{2h_0}\lt|\cl(\hat\xi+\hat A_*(\cdot,x_3))\rt|(\Omega)
	\end{equation*}
	if and only if $\hat\xi_{x_3}$ minimizes
	\begin{equation}\label{eq51}
	\mathcal{F}^*_{x_3}(\hat\xi):=\frac{1}{2}\int_{\Omega}\lt|\hat\xi\rt|^2d\hat x+\int_{\Omega}\hat\xi\cdot \hat A_*\,d\hat x,
	\end{equation}
	provided that
	\begin{equation}\label{eq52}
	\int_{\Omega}\hat\xi\cdot\hat\phi\, d\hat x\leq\frac{1}{2h_0}\lt|\cl\hat\phi\rt|(\Omega) \text{ for all } \hat\phi\in L^2(\Omega).
	\end{equation}
	It is easy to see that $\hat\nabla\cdot\hat\xi_{x_3}=0$, where $\hat\nabla=(\partial_1,\partial_2)$. Indeed, by standard Hodge decomposition, one has the orthogonal decomposition $\hat\xi_{x_3}=\hat\xi_1+\hat\xi_2$ where $\hat\nabla\cdot\hat\xi_1=0$ and $\cl\hat\xi_2=0$. It follows that
	\begin{equation*}
	\mathcal{F}_{x_3}(\hat\xi_{x_3}) = \frac{1}{2}\lVert \hat\xi_1 \rVert_{L^2(\Omega)}^2 + \frac{1}{2}\lVert \hat\xi_2 \rVert_{L^2(\Omega)}^2 + \frac{1}{2h_0}|\cl (\hat\xi_1+\hat A_*)|(\Omega) \geq \mathcal{F}_{x_3}(\hat\xi_1).
	\end{equation*}
	As $\hat\xi_{x_3}$ minimizes $\mathcal{F}_{x_3}$, it follows from the above that $\hat\xi_{x_3}=\hat\xi_1$ and hence $\hat\nabla\cdot\hat\xi_{x_3}=0$. Therefore there exists $\psi_{x_3}\in H^1(\Omega)$ such that $\hat\xi_{x_3}=(\partial_2\psi_{x_3},-\partial_1\psi_{x_3})$ and $\psi_{x_3}=0$ on $\partial\Omega$. For all $g\in L^2(\Omega)$, one can find $\hat\phi_g\in H^1(\Omega)$ such that $\cl\hat\phi_g=g$. Indeed, one can solve $\Delta f=g$ and let $\hat\phi_g:=\lt(-\partial_2 f, \partial_1 f\rt)$. Thus, it follows from \eqref{eq52} and an integration by parts that
	\begin{equation*}
	\int_{\Omega}\psi_{x_3}\,g\,d\hat x=\int_{\Omega}\lt(\partial_2\psi_{x_3},-\partial_1\psi_{x_3}\rt)\cdot\hat\phi_g\,d\hat x\leq\frac{1}{2h_0}\lVert g\rVert_{L^1(\Omega)} \text{ for all } g\in L^2(\Omega).
	\end{equation*}
	As $L^2(\Omega)$ is dense in $L^1(\Omega)$, we deduce from the above that $\psi_{x_3}\in L^{\infty}(\Omega)$ and $\lVert \psi_{x_3}\rVert_{L^{\infty}}\leq \frac{1}{2h_0}$. Plugging $\psi_{x_3}$ into \eqref{eq51} and using $\cl \hat A_*=B_*^3+1$, we obtain that $\psi_{x_3}\in H^1_0(\Omega)$ minimizes
	\begin{equation}\label{eq54}
	\min_{\lVert\psi\rVert_{\infty}\leq\frac{1}{2h_0}}\frac{1}{2}\int_{\Omega}\lt|\hat\nabla\psi\rt|^2d\hat x+\int_{\Omega}\psi\lt(B_*^3(\cdot,x_3)+1\rt)d\hat x. 
	\end{equation}

	Now we define $\hat{\ti{v}}_*(\hat x,x_3):= \hat v_{x_3}(\hat x)$ and we show that $\hat{\ti{v}}_*\in V$. First we show that $\hat{\ti{v}}_*$ is measurable. As $B_*^3+1 \in C^{\infty}(\Omega)$, by Proposition \ref{p19} in the appendix, we have $\psi_{x_3}\in W^{2,p}_0(\Omega)$ and 
	\begin{equation}\label{ep9}
	\lVert\psi_{x_3}\rVert_{W^{2,p}(\Omega)} \leq C\lt(\lVert \psi_{x_3}\rVert_{L^p(\Omega)}+\lVert B_*^3(\cdot,x_3)+1\rVert_{L^p(\Omega)}\rt)
	\end{equation}
	for all $p$ and all $0<x_3<L$. By the Sobolev embedding theorem, $\psi_{x_3}\in C^{1,\alpha}(\overline\Omega)$ and thus $\hat v_{x_3} = (\partial_2\psi_{x_3},-\partial_1\psi_{x_3})+ \hat A_*(\cdot,x_3) \in C^{0,\alpha}(\overline\Omega)$ for some $\alpha<1$.  Given $x_3\in (0,L)$ and any sequence $\{x_3^j\}$ converging to $x_3$, it is clear that $\{B_*^3(\cdot,x_3^j)\}$ converges to $B_*^3(\cdot,x_3)$ locally uniformly on $\Omega$. Thus we have $\psi_{x_3^j}\rightarrow \psi_{x_3}$ in $H^1(\Omega)$ by Proposition \ref{p18}. It then follows from (\ref{ep9}) that $\{\psi_{x_3^j}\}$ forms a bounded sequence in $W^{2,p}_0(\Omega)$. By the compact Sobolev embedding theorem, up on extraction of a subsequence, $\{\psi_{x_3^j}\}$ converges strongly to $\psi_{x_3}$ in $C^1(\overline{\Omega})$. As the whole sequence converges to $\psi_{x_3}$ in $H^1(\Omega)$, the convergence in $C^1(\overline{\Omega})$ holds true for the whole sequence. Hence $\{\hat v_{x_3^j}\}$ converges to $\hat v_{x_3}$ uniformly on $\overline\Omega$. For all $n\in \mb{N}$, define $\hat v^n(\hat x,x_3):=\hat v_{0}(\hat x)\mathbb{1}_{(0,\frac{L}{n})}(x_3)+\sum_{k=1}^{n-1}\hat v_{\frac{kL}{n}}(\hat x)\mathbb{1}_{[\frac{kL}{n},\frac{(k+1)L}{n})}(x_3)$. It is clear that $\hat v^n$ is measurable for all $n$ and $\hat v^n\rightarrow \hat{\ti{v}}_*$ a.e. in $D=\Omega\times(0,L)$, and hence $\hat{\ti{v}}_*$ is measurable. 
	
	Next, defining $h(x_3):=|\cl \hat v_{x_3}|(\Omega)$, we show that $h$ is continuous. Given $x_3\in(0,L)$ and a sequence $\{x_3^j\}$ converging to $x_3$, as $\{\hat v_{x_3^j}\}$ converges to $\hat v_{x_3}$ uniformly on $\overline\Omega$, it is clear that $\cl \hat v_{x_3^j}$ converges to $\cl\hat v_{x_3}$ weakly* as measures. By lower semicontinuity we have $h(x_3)\leq \liminf h(x_3^j)$. Now we show that $h(x_3)\geq \limsup h(x_3^j)$. We argue by contradiction. Suppose not, then there exists some $\delta_0>0$ such that $h(x_3)\leq h(x_3^{j_k})-\delta_0$ for some subsequence $\{x_3^{j_k}\}$. This translates to $|\cl\hat v_{x_3}|(\Omega)\leq |\cl\hat v_{x_3^{j_k}}|(\Omega)-\delta_0$. Using Young's inequality, we obtain
	\begin{equation*}
	\begin{split}
	&F_{x_3^{j_k}}(\hat v_{x_3})=\frac{1}{2}\lt\lVert \hat v_{x_3}-\hat A_*(\cdot,x_3^{j_k}) \rt\rVert_{L^2(\Omega)}^2+\frac{1}{2h_0}\lt|\cl \hat v_{x_3}\rt|(\Omega)\\
	&\leq \frac{1}{2}(1+\sigma)\lt\lVert \hat v_{x_3^{j_k}}-\hat A_*(\cdot,x_3^{j_k}) \rt\rVert_{L^2(\Omega)}^2 + C(\sigma) \lt\lVert \hat v_{x_3}-\hat v_{x_3^{j_k}} \rt\rVert_{L^2(\Omega)}^2+\frac{1}{2h_0}\lt|\cl \hat v_{x_3^{j_k}}\rt|(\Omega)-\frac{\delta_0}{2h_0}\\
	&=F_{x_3^{j_k}}(\hat v_{x_3^{j_k}})+\frac{\sigma}{2}\lt\lVert \hat v_{x_3^{j_k}}-\hat A_*(\cdot,x_3^{j_k}) \rt\rVert_{L^2(\Omega)}^2 + C(\sigma) \lt\lVert \hat v_{x_3}-\hat v_{x_3^{j_k}}\rt\rVert_{L^2(\Omega)}^2-\frac{\delta_0}{2h_0}
	\end{split}
	\end{equation*}
	for all $\sigma>0$ and some constant $C(\sigma)$ depending only on $\sigma$. Note that $\{\hat v_{x_3^{j_k}}-\hat A_*(\cdot,x_3^{j_k})\}$ converges to $\hat v_{x_3}-\hat A_*(\cdot,x_3)$ uniformly on $\overline\Omega$ and thus $\lVert \hat v_{x_3^{j_k}}-\hat A_*(\cdot,x_3^{j_k}) \rVert_{L^2(\Omega)}^2$ is bounded independent of $k$ for $k$ sufficiently large. Thus by letting $\sigma\rightarrow 0$ and then $k\rightarrow\infty$ we observe that $F_{x_3^{j_k}}(\hat v_{x_3})<F_{x_3^{j_k}}(\hat v_{x_3^{j_k}})$ for all $k$ sufficiently large, which is a contradiction as $\hat v_{x_3^{j_k}}$ is the minimizer of $F_{x_3^{j_k}}$. This shows that $\limsup h(x_3^j)\leq h(x_3)\leq \liminf h(x_3^j)$ for all sequences $\{x_3^j\}$ converging to $x_3$, and thus $h$ is continuous in $(0,L)$. In particular, $h$ is measurable. Using Lemma \ref{l13}, we have that
	\begin{equation}\label{ep10}
	\begin{split}
	&\frac{1}{2}\int_{D} \lt|\hat v_* - \hat A_*\rt|^2\, dx + \frac{1}{2h_0}\lt|\cl\hat v_*\rt|(D) \\
	&\qd\qd\qd= \int_{0}^{L} \lt(\frac{1}{2}\int_{\Omega} \lt|\hat v_*(\hat x,x_3) - \hat A_*(\hat x,x_3)\rt|^2\, d\hat x + \frac{1}{2h_0}\lt|\cl\hat v_*(\cdot,x_3)\rt|(\Omega)\rt) dx_3\\
	&\qd\qd\qd \geq \int_{0}^{L} \lt(\frac{1}{2}\int_{\Omega} \lt|\hat v_{x_3}(\hat x) - \hat A_*(\hat x,x_3)\rt|^2\, d\hat x + \frac{1}{2h_0}\lt|\cl\hat v_{x_3}\rt|(\Omega)\rt) dx_3.
	\end{split}
	\end{equation}
	It is clear from the above that $\hat{\ti v}_*\in L^2(D;\R^2)$ and $\int_{0}^{L}h(x_3)dx_3<\infty$. Further, it is straightforward to see that $\lt|\cl \hat{\ti v}_*\rt|(D)\leq \int_{0}^{L}h(x_3)dx_3$ and thus $\hat{\ti v}_*\in V$. Then (\ref{ep10}) becomes
	\begin{equation*}
	\frac{1}{2}\int_{D} \lt|\hat v_* - \hat A_*\rt|^2\, dx + \frac{1}{2h_0}\lt|\cl \hat v_*\rt|(D) \geq \frac{1}{2}\int_{D} \lt|\hat{\ti{v}}_* - \hat A_*\rt|^2\, dx + \frac{1}{2h_0}\lt|\cl\hat{\ti v}_*\rt|(D). 
	\end{equation*}
	As $\hat v_*$ is the unique minimizer of $F$ (given in (\ref{ep11})) in $V$, it follows that $\hat{\ti v}_* = \hat v_*$. 
	
	Using Lemma \ref{l13}, we have that $\lt|\cl\hat v_*\rt|(D)=0$ if and only if $\lt|\cl\hat v_{x_3}\rt|(\Omega)=0$ for a.e. $x_3\in(0,L)$. Given $x_3\in(0,L)$, recall that $\hat v_{x_3}=(\partial_2\psi_{x_3},-\partial_1\psi_{x_3})+\hat A_*(\cdot,x_3)$ and $\vec{B}_*=\Cl(\vec{A}_*-\vec{a})$, and thus $\cl \hat v_{x_3}=-\hat\Delta\psi_{x_3}+B_*^3(\cdot,x_3)+1$ as measures, where $\hat\Delta$ is the $2$d Laplacian. By standard theory about the obstacle problem (\ref{eq54}), $\cl\hat v_{x_3}=0$ if and only if $-\hat\Delta\psi_{x_3}+B_*^3(\cdot,x_3)+1=0$ if and only if the unique solution $\ti\psi_{x_3}$ of the following problem
	\begin{equation*}
	\begin{cases}
	-\hat\Delta\psi+(B_*^3(\cdot,x_3)+1)=0 &\text{in }\Omega,\\
	\psi=0 &\text{on } \partial\Omega
	\end{cases}
	\end{equation*}
	satisfies $\lVert\ti\psi_{x_3}\rVert_{\infty}\leq\frac{1}{2h_0}$ and thus $\psi_{x_3}=\ti\psi_{x_3}$. Therefore $\cl\hat v_*=0$ if and only if $\lVert\psi_{x_3}\rVert_{\infty}\leq\frac{1}{2h_0}$ for a.e. $x_3\in(0,L)$. This is further equivalent to $\lVert\psi_{x_3}\rVert_{\infty}\leq\frac{1}{2h_0}$ for all $x_3\in(0,L)$ because of the convergence of $\{\psi_{x_3^j}\}$ to $\psi_{x_3}$ in $C^1(\overline\Omega)$ provided $x_3^j\rightarrow x_3$. The conclusion of Theorem \ref{t4} hence follows from this and Lemma \ref{l013}.
\end{proof}

%
%

\section{Appendix}\label{appendix}

In this appendix, we give the proof of some regularity for the double obstacle problem (\ref{eq54}) in the proof of Theorem \ref{t4}. We consider a slightly more general problem. Let $a_1<0<a_2$ be two constants, and denote
\begin{equation}\label{ep3}
\mathbb{K}:=\lt\{u\in H^1_0(\Omega): a_1\leq u\leq a_2\text{ a.e. in }\Omega\rt\}. 
\end{equation}
It is clear that $\mb{K}$ is a closed, nonempty and convex subset of $H^1_0(\Omega)$. Let $f\in L^2(\Omega)$ and consider the minimization problem
\begin{equation}\label{ep1}
\min_{u\in \mb{K}}\lt(\frac{1}{2}\int_{\Omega}\lt|\nabla u\rt|^2\,dx-\int_{\Omega}u f\,dx\rt). 
\end{equation}
It is standard that $u$ solves the problem (\ref{ep1}) if and only if $u$ solves the variational inequality
\begin{equation}\label{ep2}
\int_{\Omega}\nabla u\cdot\nabla (v-u)dx \geq \int_{\Omega} f\lt(v-u\rt)dx \text{ for all } v\in\mb{K}.
\end{equation}
By the Lions-Stampacchia theorem (see Theorem 3.1 in \cite{ro}, page 93) we have
\begin{proposition}\label{p18}
	Let $\Omega\subset\R^2$ be a bounded domain. Then for each $f\in L^2(\Omega)$, there exists a unique solution $u\in \mb{K}$ to the variational inequality (\ref{ep2}). Further let $f_1, f_2\in L^2(\Omega)$ and $u_1, u_2\in \mb{K}$ be the solution of the variational inequality (\ref{ep2}) with data $f_1, f_2$ respectively, then
	\begin{equation*}
	\lVert u_1-u_2\rVert_{H^1}\leq C\lVert f_1-f_2\rVert_{L^2}
	\end{equation*}
	for some constant $C$ independent of $f_1, f_2$.
\end{proposition}
The $C^{1,\alpha}$ regularity of solutions to single obstacle problems is well-known (see, e.g., \cite{ro}). On the other hand, the literature on double obstacle problems seems to be limited, although similar regularity results have been established (see, e.g., \cite{dmv}, \cite{kz}, \cite{du}). What we need is a strong dependence on the data of the solution to the problem, which should be well-known to experts. However, we were not able to find an explicit reference on this result. So we provide a proof by slightly modifying the proof for single obstacle problems (see Chapter 5 in \cite{ro}) for the convenience of the reader.

\begin{proposition}\label{p19}
	Let $\Omega\subset\R^2$ be a bounded smooth domain. Let $\mb{K}$ be defined in (\ref{ep3}) and $f\in L^p(\Omega)$ for some $p\geq 2$. Then the solution $u\in\mb{K}$ of the variational inequality (\ref{ep2}) satisfies $u\in W_0^{2,p}(\Omega)$ and 
	\begin{equation}\label{ep6}
	\lVert u\rVert_{W^{2,p}(\Omega)}\leq C\lt(\lVert u\rVert_{L^p(\Omega)}+\lVert f\rVert_{L^p(\Omega)}\rt)
	\end{equation}
	for some constant $C$ independent of $f$.
\end{proposition}
We slightly modify the duality argument for single obstacle problem as in Chapter 5, \cite{ro}. We need the following lemmas.

\begin{lemma}\label{l20}
	Let $u\in\mb{K}$ be the solution of the variational inequality (\ref{ep2}). Then we have $-f^-\leq -\Delta u\leq f^+$ in $H^{-1}$, where $f^+=\max\{f,0\}$ and $f^-=\max\{-f,0\}$ are the positive and negative parts of $f$, respectively.
\end{lemma}

\begin{proof}
	To show $-f^-\leq -\Delta u$ in $H^{-1}$, we consider $y\in H^1_0(\Omega)$ satisfying 
	\begin{equation}\label{ep4}
	y\geq u\text{ a.e.},\qd \int_{\Omega}\nabla y\cdot\nabla\lt(w-y\rt)dx\geq -\int_{\Omega}f^-\lt(w-y\rt)dx \text{ for all } w\in H^1_0(\Omega), w\geq u.
	\end{equation}
	This is a single obstacle problem. The existence and uniqueness of solutions is well-known. We claim that $y=u$. Then taking $w=u+v$ for any $v\in H^1_0(\Omega)$ and $v\geq 0$, it follows that
	\begin{equation*}
	\int_{\Omega} \nabla u\cdot \nabla vdx\geq -\int_{\Omega}f^- vdx \text{ for all } v\in H^1_0(\Omega), v\geq 0, 
	\end{equation*}
	which is exactly what we need to conclude that $-f^-\leq -\Delta u$ in $H^{-1}$.
	
	We first show that $y\leq a_2$ a.e. in $\Omega$. Let $w=y-\lt(y-a_2\rt)^+\geq u$. Plugging $w$ into (\ref{ep4}) gives
	\begin{equation*}
	-\int_{\Omega}\nabla y\cdot\nabla\lt(y-a_2\rt)^+dx\geq \int_{\Omega}f^-\lt(y-a_2\rt)^+dx.
	\end{equation*}
	Also noting that $\nabla a_2=0$, we have
	\begin{equation*}
	\int_{\Omega}\nabla a_2\cdot\nabla\lt(y-a_2\rt)^+dx\geq -\int_{\Omega}f^-\lt(y-a_2\rt)^+dx.
	\end{equation*}
	Adding the above two inequalities gives
	\begin{equation*}
	\int_{\Omega}\lt|\nabla\lt(y-a_2\rt)^+\rt|^2dx=\int_{\Omega}\nabla\lt(y-a_2\rt)\cdot\nabla\lt(y-a_2\rt)^+dx\leq 0
	\end{equation*}
	and thus $\lt(y-a_2\rt)^+=0$ a.e., which implies $y\leq a_2$ a.e. in $\Omega$.
	
	Next we show that $y\leq u$. Let $v=u+\lt(y-u\rt)^+\geq u$. Then $a_1\leq v\leq a_2$. Plugging $v$ into (\ref{ep2}) we obtain
	\begin{equation*}
	\int_{\Omega}\nabla u\cdot\nabla (y-u)^+dx \geq \int_{\Omega} f(y-u)^+dx.
	\end{equation*}
	Take $w=y-(y-u)^+\geq u$. Plugging $w$ into (\ref{ep4}) gives
	\begin{equation*}
	-\int_{\Omega}\nabla y\cdot\nabla(y-u)^+dx\geq \int_{\Omega}f^-(y-u)^+dx.
	\end{equation*}
	Adding the above two inequalities gives
	\begin{equation*}
	\begin{split}
	&\int_{\Omega}\lt|\nabla\lt(y-u\rt)^+\rt|^2dx=\int_{\Omega}\nabla\lt(y-u\rt)\cdot\nabla\lt(y-u\rt)^+dx\\
	&\qd\qd\qd\qd\leq -\int_{\Omega}\lt(f+f^-\rt)(y-u)^+dx = -\int_{\Omega} f^{+}(y-u)^+dx\leq 0,
	\end{split}
	\end{equation*}
	and thus $y\leq u$ a.e. in $\Omega$. Hence we conclude that $y=u$ and as discussed above, it follows that $-f^-\leq -\Delta u$ in $H^{-1}$.
	
	To show that $-\Delta u\leq f^+$ in $H^{-1}$, we consider $z\in H^1_0(\Omega)$ such that
	\begin{equation}\label{ep5}
	z\leq u\text{ a.e.},\qd \int_{\Omega}\nabla z\cdot\nabla\lt(w-z\rt)dx\geq \int_{\Omega}f^+\lt(w-z\rt)dx \text{ for all } w\in H^1_0(\Omega), w\leq u.
	\end{equation}
	Following the same lines as above (see the proof of Theorem 2.1 in Chapter 5, \cite{ro} for details), one can show that $z=u$. Taking $w=u-v$ for all $v\in H^1_0(\Omega)$ and $v\geq 0$ in (\ref{ep5}) gives $-\Delta u\leq f^+$ in $H^{-1}$.
\end{proof}

\begin{lemma}\label{l21}
	Let $\Omega\subset\R^2$ be a bounded smooth domain. Let $\mb{K}$ be defined in (\ref{ep3}) and $f\in L^2(\Omega)$. Then the solution $u\in\mb{K}$ of the variational inequality (\ref{ep2}) satisfies $u\in H^2(\Omega)$.
\end{lemma}

\begin{proof}
	We denote by $H=L^2(\Omega)$ and $V=H^1_0(\Omega)$. Further we denote by $\langle\cdot,\cdot\rangle_{H}$ and $\langle\cdot,\cdot\rangle_{V}$ the duality between $H'$ and $H$ and between $V'$  and $V$, respectively. Using Lemma \ref{l20}, we have $-\Delta u-f\leq f^+-f=f^-$ and $-\Delta u-f\geq -f^--f=-f^+$ in $V'$. Now for all $v\in V$, using the Cauchy-Schwartz inequality, we have
	\begin{equation}\label{ep7}
	\begin{split}
	\langle -\Delta u - f, v\rangle_V &= \langle -\Delta u - f, v^+\rangle_V - \langle -\Delta u - f, v^-\rangle_V\\
	&\leq \langle f^-, v^+\rangle_H + \langle f^+, v^-\rangle_H \\
	&\leq \lVert f^-\rVert_H\lVert v^+\rVert_H+\lVert f^+\rVert_H\lVert v^-\rVert_H\leq \lVert f\rVert_H\lVert v\rVert_H.
	\end{split}
	\end{equation}
	Similarly we have
	\begin{equation}\label{ep8}
	\begin{split}
	\langle -\Delta u - f, v\rangle_V &= \langle -\Delta u - f, v^+\rangle_V - \langle -\Delta u - f, v^-\rangle_V\\
	&\geq -\langle f^+, v^+\rangle_H - \langle f^-, v^-\rangle_H \\
	&\geq -\lVert f^+\rVert_H\lVert v^+\rVert_H-\lVert f^-\rVert_H\lVert v^-\rVert_H\geq -\lVert f\rVert_H\lVert v\rVert_H.
	\end{split}
	\end{equation}
	As $V$ is dense in $H$, $-\Delta u-f$ can be extended to be a bounded linear operator on $H$ and thus $-\Delta u\in L^2(\Omega)$. Further we have $\lVert -\Delta u\rVert_{L^2(\Omega)}\leq 2\lVert f\rVert_{L^2(\Omega)}$. It follows from standard elliptic regularity (see, e.g., Theorem 8.12 in \cite{GT}) that $u\in H^2(\Omega)$.
\end{proof}

\begin{proof}[Proof of Proposition \ref{p19}]
	From Lemma \ref{l21} we know $u\in H^2(\Omega)$, and thus $-f^{-}\leq -\Delta u-f\leq f^+$ a.e. in $\Omega$ by Lemma \ref{l20}. We claim that $-\Delta u\in L^p(\Omega)$. To this end, we carry out arguments similar to those in (\ref{ep7}) and (\ref{ep8}) to find that for all $v\in L^2(\Omega)\subset L^{p'}(\Omega)$ for $\frac{1}{p}+\frac{1}{p'}=1$, we have
	\begin{equation*}
	\int_{\Omega}\lt(-\Delta u-f\rt)vdx \leq \int_{\Omega}\lt(f^-v^++f^+v^-\rt)dx\leq \int_{\Omega}|f||v|dx\leq\lVert f\rVert_{L^p}\lVert v\rVert_{L^{p'}}
	\end{equation*}
	and 
	\begin{equation*}
	\int_{\Omega}\lt(-\Delta u-f\rt)vdx \geq -\int_{\Omega}\lt(f^+v^++f^-v^-\rt)dx\geq -\int_{\Omega}|f||v|dx\geq-\lVert f\rVert_{L^p}\lVert v\rVert_{L^{p'}}.
	\end{equation*}
	As $L^2(\Omega)$ is dense in $L^{p'}(\Omega)$, it follows that $-\Delta u-f$ can be extended to be a bounded linear operator on $L^{p'}(\Omega)$ and thus $-\Delta u\in L^p(\Omega)$ and $\lVert -\Delta u\rVert_{L^p}\leq 2\lVert f\rVert_{L^p}$. It follows from standard elliptic regularity (see, e.g., Theorem 9.13 in \cite{GT}) that $u\in W^{2,p}(\Omega)$ and the estimate (\ref{ep6}) holds true.
\end{proof}

\end{document}